\documentclass[12pt,twoside,reqno]{amsart}
\usepackage{amsmath}
\usepackage{amsfonts}
\usepackage{amssymb}
\usepackage{color}
\usepackage{mathrsfs}
\usepackage{times}
\usepackage{graphicx,psfrag,epsfig}

\newcommand{\NN}{\mathbb N} 
\newcommand{\RR}{\mathbb R}

\newcommand{\ml}{\mathcal{L}}

\DeclareMathAlphabet{\mathpzc}{OT1}{pzc}{m}{it}
\newtheorem{theorem}{Theorem}[section]
\newtheorem{corollary}[theorem]{Corollary}
\newtheorem{lemma}[theorem]{Lemma}
\newtheorem{proposition}[theorem]{Proposition}
\newtheorem{definition}[theorem]{Definition}
\newtheorem{remark}[theorem]{Remark}
\numberwithin{equation}{section}

\newcommand{\unb}{\mathbb{B}_1}
\newcommand{\cunb}{\bar{\mathbb{B}}_1}
\newcommand{\uns}{\partial\mathbb{B}_1}
\newcommand{\rd}{\mathrm{d}}
\begin{document}
\title{A fourth-order model for MEMS with\\ clamped boundary conditions} 

\author{Philippe Lauren\c{c}ot}
\address{Institut de Math\'ematiques de Toulouse, CNRS UMR~5219, Universit\'e de Toulouse \\ F--31062 Toulouse Cedex 9, France}
\email{laurenco@math.univ-toulouse.fr}

\author{Christoph Walker}
\address{Leibniz Universit\"at Hannover\\ Institut f\" ur Angewandte Mathematik \\ Welfengarten 1 \\ D--30167 Hannover\\ Germany}
\email{walker@ifam.uni-hannover.de}


\date{\today}

%%%%%%%%%%%%%%%%%%%%%%%%%%%%%%%%%%%%%%%%%%%%%%%%%%%%%%%%%%%%%%%%%%%%
\begin{abstract}
The dynamical and stationary behaviors of a fourth-order equation in the unit ball with clamped boundary conditions and a singular reaction term are investigated.  The equation arises in the modeling of microelectromechanical systems (MEMS) and includes a positive voltage parameter $\lambda$. It is shown that there is a threshold value $\lambda_*>0$ of the voltage parameter such that no radially symmetric stationary solution exists for $\lambda>\lambda_*$, while at least two  such solutions exist for $\lambda\in (0,\lambda_*)$. Local and global well-posedness results are obtained for the corresponding hyperbolic and parabolic evolution  problems as well as the occurrence of finite time singularities  when $\lambda>\lambda_*$.
\end{abstract}
%%%%%%%%%%%%%%%%%%%%%%%%%%%%%%%%%%%%%%%%%%%%%%%%%%%%%%%%%%%%%%%%%%%%

\maketitle

%
%     HEADLINES
%
\pagestyle{myheadings}
\markboth{\sc{Ph. Lauren\c{c}ot and Ch. Walker}}{\sc{A fourth-order MEMS model with clamped boundary conditions}}

%%%%%%%%%%%%%%%%%%%%%%%%%%%%%%%%%%%%%%%%%%%%%%%%%%%%%%%%%%%%%%%%%%%%
%%%%%%%%%%%%%%%%%%%%%%%%%%%%%%%%%%%%%%%%%%%%%%%%%%%%%%%%%%%%%%%%%%%%
\section{Introduction}\label{sec1}
%%%%%%%%%%%%%%%%%%%%%%%%%%%%%%%%%%%%%%%%%%%%%%%%%%%%%%%%%%%%%%%%%%%%
%%%%%%%%%%%%%%%%%%%%%%%%%%%%%%%%%%%%%%%%%%%%%%%%%%%%%%%%%%%%%%%%%%%%

Electrostatically actuated microelectromechanical systems (MEMS) are microscopic devices which combine mechanical and electrostatic effects. A typical MEMS device is made of a rigid conducting ground plate above which a clamped deformable membrane coated with a thin conducting film is suspended. Application of a voltage difference induces a Coulomb force which, in turn, generates a displacement of the membrane. An ubiquitous feature of such devices is, that when the applied voltage exceeds a certain threshold value, the membrane might collapse (or touch down) on the ground plate. Controlling the occurrence of this phenomenon -- usually referred to as the ``pull-in'' instability -- is of utmost practical importance in the design of such devices either to set up optimal operating conditions or to avoid device damaging. Mathematical models have been derived to describe MEMS devices which lead to free boundary problems due to the  deformable membrane \cite{BP07}. Since these models are difficult to analyze mathematically (though recent contributions can be found in \cite{Ci07, ELW1, LW13}), one often takes advantage of the small aspect ratio of the devices to reduce the free boundary problem to a single equation for the displacement, see \cite{BP07}. More precisely, the small aspect ratio model describing the dynamics of the displacement $u=u(t,x)$ of the membrane $\Omega\subset\RR^d$ reads
\begin{align}
\gamma^2\partial_t^2 u+\partial_t u +B\Delta^2 u- T\Delta u&=-\frac{\lambda}{(1+u)^2}\ , &t>0\ , & & x\in \Omega\ ,\label{hyper1}\\
u=\partial_\nu u&=0\ ,& t>0\ ,& & x\in \partial\Omega\ ,\label{hyper2}\\
u(0,\cdot)=u^0\ ,\quad \partial_t u(0,\cdot)&= u^1\ ,& && x\in \Omega\ .\label{hyper3}
\end{align}
Here, $\gamma^2\partial_t^2 u$ and $\partial_t u$ account, respectively, for inertia and damping effects, $B\Delta^2 u$ and $- T\Delta u$ are due to bending and stretching of the membrane, while $-\lambda(1+u)^{-2}$ reflects the action of the electrostatic forces in the small aspect ratio limit. The parameter $\lambda$ is proportional to the square of the applied voltage. Observe that the right-hand side of \eqref{hyper1} features a singularity when $u=-1$, which corresponds to the touchdown phenomenon already mentioned above. Since the strength of the singular reaction term is tuned by the parameter $\lambda$, it is not surprising that the latter governs the existence of stationary solutions, that is, solutions to
\begin{align}
B\Delta^2 u- T\Delta u&=-\frac{\lambda}{(1+u)^2}\ , & x\in \Omega\ ,\label{hyper1s}\\
u=\partial_\nu u&=0\ ,&  x\in \partial\Omega\ .\label{hyper2s}
\end{align}
When bending is neglected, that is, when  $B=0$, this problem reduces to a second-order elliptic equation that has been studied extensively in the recent past, see e.g. the monograph \cite{EGG10} and the references therein. As expected from the physics there is a critical value $\lambda_*>0$ such that no stationary solution exists if $\lambda>\lambda_*$ and at least one stationary solution exists for $\lambda\in (0,\lambda_*)$.
Let us emphasize that the comparison principle is available in this case and turns out to be a key tool for the analysis. Less attention has been dedicated to \eqref{hyper1s}-\eqref{hyper2s} with $B>0$, one reason might be the lack of a maximum principle in general for the clamped boundary conditions \eqref{hyper2s} (also called Dirichlet boundary conditions), see the monograph~\cite{GGS10} for a detailed discussion of positivity properties of higher-order operators. Recall that the situation is completely different if the clamped boundary conditions \eqref{hyper2} are replaced with pinned (or Navier) boundary conditions
\begin{equation}\label{pinned}
u=\Delta u=0\ \text{ on }\ \partial\Omega\ ,
\end{equation}
since in this case the maximum principle holds in arbitrary domains \cite{EGG10, GGS10}. This allows one in particular to show similar results for the  fourth-order problem \eqref{hyper1s}, \eqref{pinned} as outlined above for the second-order case corresponding to $B=0$.  
We refer to \cite{EGG10, LY07} for details.

\bigskip

Returning to the case of clamped boundary conditions, when $B>0$ existence of  solutions to \eqref{hyper1s}-\eqref{hyper2s} for small values of $\lambda$ has been established in \cite{LY07} for an arbitrary domain $\Omega$. This is the only result we are aware of for a general domain. In the particular case when $\Omega$ equals the unit ball $\unb$, Boggio \cite{Bo05} has uncovered the availability of the maximum principle for the operator $B\Delta^2$ with boundary conditions \eqref{hyper2s} by showing that the corresponding Green function is positive. This fact has been used in \cite{EGG10} to describe more precisely the set of solutions to \eqref{hyper1s}-\eqref{hyper2s} with $T=0$. It has actually been shown in \cite[Chapter 11]{EGG10} that there is a critical threshold value $\lambda_*>0$ such that no  solution exists for $\lambda>\lambda_*$ and a solution exists for $\lambda\in (0,\lambda_*]$.  

Very recently we were able to extend Boggio's maximum principle to the operator $B\Delta^2-T\Delta$ with $T>0$ and boundary conditions \eqref{hyper2s} in the class of radially symmetric functions in $\unb$ \cite{LWxx}. Taking advantage of this property not only allows us to extend the results of \cite{EGG10} to include $T>0$ (for $d=1,2$), but also to show that for each voltage value $\lambda\in (0,\lambda_*)$ there are at least two (radially symmetric) solutions to \eqref{hyper1s}-\eqref{hyper2s}, thereby answering a question raised in \cite[p.~268]{EGG10}. A summary of our results on radially symmetric solutions
to \eqref{hyper1s}-\eqref{hyper2s} is stated in the following theorem.

%%%%%%%%%%%%%%%%%%%%%%%%%%%%%%%%%%%%%%%%%%%%%%%%%%%%%%%%%%%%%%%%%%%%%%%
%%%%%%%%%%%%%%%%%%%%%%%%%%%%%%%%%%%%%%%%%%%%%%%%%%%%%%%%%%%%%%%%%%%%%%%
\begin{theorem}\label{TSSIntroduction}
Let $d=1,2$, $\Omega=\unb$, $B>0$, and $T\ge 0$. There exists $\lambda_*>0$ such that there is no radially symmetric solution to \eqref{hyper1s}-\eqref{hyper2s} for $\lambda>\lambda_*$. Moreover, there is a continuous  curve $(\Lambda(s), U(s))$, $s\in [0,\infty)$ in $\RR\times C^4(\cunb)$ such that $U(s)$ is for each $s\in [0,\infty)$  a radially symmetric solution to \eqref{hyper1s}-\eqref{hyper2s} with $\lambda=\Lambda(s)$. Moreover,  $(\Lambda(0), U(0))=(0, 0)$ and $(\Lambda(s), U(s))\rightarrow (0,\omega)$ as $s\rightarrow \infty$, where $\omega$ is an explicitly given radially symmetric function with $\omega(0)=-1$. Finally, there is $s_*>0$ such that $\Lambda$ is an increasing function from $[0,s_*]$ onto $[0,\lambda_*]$ and $\Lambda$ is decreasing in a right-neighborhood of $s_*$.
\end{theorem}
%%%%%%%%%%%%%%%%%%%%%%%%%%%%%%%%%%%%%%%%%%%%%%%%%%%%%%%%%%%%%%%%%%%%%%%
%%%%%%%%%%%%%%%%%%%%%%%%%%%%%%%%%%%%%%%%%%%%%%%%%%%%%%%%%%%%%%%%%%%%%%%

\begin{remark}\label{ghoussoub}
Theorem~\ref{TSSIntroduction} guarantees that for each $\lambda\in (0,\lambda_*)$, there are at least two (radially symmetric) solutions  to \eqref{hyper1s}-\eqref{hyper2s}. More precisely, for each $\lambda\in (0,\lambda_*)$ there are at least two values $0<s_1<s_*<s_2$ with $\Lambda(s_j)=\lambda$ for $j=1,2$ and $U(s_2)\le U(s_1)$ in $\unb$ with $U(s_2)\not= U(s_1)$. 
\end{remark}

Let us mention here that the construction of the curve $(\Lambda(s), U(s))$, $s\in [0,s_*]$ follows the lines of \cite[Chapter 11]{EGG10}, where a similar result is proved when $T=0$. There are, however, some technical difficulties to be overcome. Nevertheless, we emphasize that the main contributions of Theorem~\ref{TSSIntroduction} are the extension of the curve past $(\Lambda(s_*), U(s_*))$ and the identification of its end point $\omega$ as $s\rightarrow \infty$. Interestingly, the end point $\omega$ is given as a solution of a boundary value problem in $\unb\setminus\{0\}$ which can be computed explicitly (see Theorem~\ref{C217} below); a plot of $\omega$ is shown in Figure~\ref{fig1}. The qualitative behavior of $\omega$ is the same for $d=1$ and $d=2$. 

For the case of pinned boundary conditions \eqref{pinned} it has been shown in \cite[Theorem 1.2]{GW09} with the help of the Mountain Pass Principle that there are at least two solutions for $\lambda\in (0,\lambda_*)$. The limit as $\lambda\rightarrow 0$ of the minimum of the solutions constructed with the  Mountain Pass Principle is proved to be $-1$. However, the precise profile as $\lambda\rightarrow 0$ is not identified therein.

\medskip

%------------
\begin{figure}
\centering\includegraphics[width=10cm]{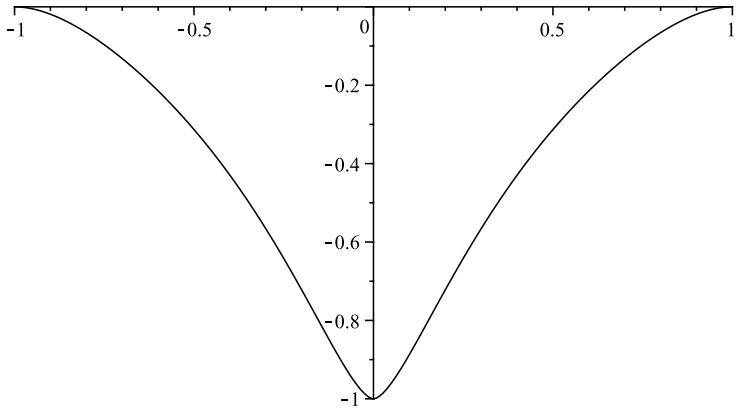}
\caption{\small Plot of the function $\omega$ for $d=2$ and $(B,T)=(1,50)$.}\label{fig1}
\end{figure}
%------------

The proof of Theorem~\ref{TSSIntroduction} is performed in Section~\ref{sec2}. Therein we give a more detailed characterization of the set of radially symmetric stationary solutions. Actually, the implicit function theorem provides a branch $\mathcal{A}_0$ of radially symmetric solutions $(\lambda,u)$ to \eqref{hyper1s}-\eqref{hyper2s} emanating from $(0,0)$. We then use the  bifurcation theory of \cite{BDT00} for real analytic functions to extend $\mathcal{A}_0$ to a global curve $\mathcal{A}$ (see Theorem~\ref{T19} below).  Next we show that $\mathcal{A}_0$ coincides with the branch of stable radially symmetric stationary solutions (see Corollary~\ref{C20}). To achieve this result, the maximum principle obtained in \cite{LWxx} is essential as was Boggio's maximum principle in \cite{EGG10} for the case $T=0$. The outcome of this analysis is that there is a threshold value $\lambda_*>0$ such that there is no radially symmetric stationary solution for $\lambda>\lambda_*$, while for any $\lambda\in (0,\lambda_*)$ there is a unique stationary solution $u_\lambda$ such that $(\lambda,u_\lambda)\in \mathcal{A}_0$. We then show that for $\lambda=\lambda_*$ there is also a radially symmetric stationary solution $u_{\lambda_*}$, which guarantees on the one hand that $\mathcal{A}_0\not=\mathcal{A}$ and on the other hand that we may apply the result of \cite{CR73} to extend the branch $\mathcal{A}_0$ to the ``right'' of $(\lambda_*, u_{\lambda_*})$ (see Theorem~\ref{T21}). The final step is to show that $\mathcal{A}$ connects  $(\lambda,u)=(0,0)$ to the end point $(0,\omega)$ and to identify the latter (see Theorem~\ref{C217}). As a consequence, the continuous curve $\mathcal{A}$ passes through  $(\lambda,u)=(\lambda_*,u_{\lambda_*})$, which implies Remark~\ref{ghoussoub}. 

\bigskip

We shall also investigate local and global well-posedness of the dynamic problem \eqref{hyper1}-\eqref{hyper3}. It is worth pointing out that the maximum principle, which is at the heart of the proof of Theorem~\ref{TSSIntroduction} and the main reason to restrict the analysis to $\unb$, is no longer valid for the time-dependent problem ~\eqref{hyper1}-\eqref{hyper3}. In order to construct solutions to the evolution problem, we therefore have to employ an alternative method which does not rely on maximum principles. Our approach is based on semigroup theory and is not specific to $\unb$.
We thus consider an arbitrary domain $\Omega\subset\RR^d$, $d=1,2$, in the following and begin with the hyperbolic problem which has not received much attention so far.

%%%%%%%%%%%%%%%%%%%%%%%%%%%%%%%%%%%%%%%%%%%%%%%%%%%%%%%%%%%%%%%%%%%%
\begin{theorem}\label{ThyperIntroduction}
Let $\Omega\subset\mathbb{R}^d$ be an arbitrary bounded smooth domain for $d=1,2$ and $B>0$, $T\ge 0$. Let $\lambda>0$ and $\kappa\in (0,1)$. Let $(u^0,u^1)\in H^4(\Omega)\times H^2(\Omega)$ be such that $u^0\ge -1+\kappa$ in~$\Omega$ and such that $u^0$ and $u^1$ both satisfy the boundary  conditions \eqref{hyper2}. Then the following hold:

\begin{itemize}
\item[(i)] There are $\tau_m>0$ and a unique maximal solution $u$ to \eqref{hyper1}-\eqref{hyper3} with regularity 
$$
u\in C([0,\tau_m),H^2(\Omega))\cap C^1([0,\tau_m),L_2(\Omega))\ ,\quad \partial_t^k u\in L_1(0,\tau; H^{4-2k}(\Omega))
$$
for $k=0,1,2$ and $\tau\in (0,\tau_m)$.

\item[(ii)] If $\tau_m<\infty$, then  $$\liminf_{t\rightarrow\tau_m}\big(\min_{\bar\Omega} u(t)\big)=-1\ .$$%$(1+u)^{-2}\not\in L_\infty (0,\tau_m; L_2(\Omega))$.

\item[(iii)] There are $\lambda_1(\kappa)>0$ and $r(\kappa)>0$ such that $\tau_m=\infty$ provided that $\lambda\le \lambda_1(\kappa)$ and $\|(u^0,u^1)\|_{H^2\times L_2}\le r(\kappa)$. In this case, $u\in L_\infty(0,\infty;H^2(\Omega))$ and 
$$
\inf_{(0,\infty)\times \Omega} u>-1\ .
$$ 

\item[(iv)] If $\Omega=\unb$ and both $u^0$ and $u^1$ are radially symmetric, then so is $u(t)$ for each $t\in [0,\tau_m)$.
\end{itemize}
\end{theorem}
%%%%%%%%%%%%%%%%%%%%%%%%%%%%%%%%%%%%%%%%%%%%%%%%%%%%%%%%%%%%%%%%%%%%

 Similar results have been established in \cite{KLNT11} for $d=1$ and $B=0$ (without the damping term $\partial_t u$) and in \cite{Gu10} for the pinned boundary conditions \eqref{pinned} and $d\in\{1,2,3\}$. Let us point out that the semigroup approach allows us to obtain strong solutions instead of weak solutions as in \cite{Gu10, KLNT11}, see Section~\ref{sec31} for the proof of Theorem~\ref{ThyperIntroduction}.

\bigskip

In the damping dominated limit $\gamma \ll 1$ when viscous forces dominate over inertial forces, \eqref{hyper1}-\eqref{hyper3} reduces to the parabolic problem
\begin{align}
\partial_t u +B\Delta^2 u- T\Delta u&=-\frac{\lambda}{(1+u)^2}\ ,\quad &t>0\ ,\quad & x\in \Omega\ ,\label{hyper1p}\\
u=\partial_\nu u&=0\ , & t>0\ ,\quad & x\in \partial\Omega\ ,\label{hyper2p}\\
u(0,\cdot)&=u^0\ , & & x\in \Omega\ .\label{hyper3p}
\end{align}
To the best of our knowledge, this problem has not been investigated so far. For this reason, we include a result on its well-posedness though local existence of solutions is a rather classical argument. To obtain global solutions for small values of $\lambda$, we consider only regular initial values in the next theorem for the sake of simplicity.

%%%%%%%%%%%%%%%%%%%%%%%%%%%%%%%%%%%%%%%%%%%%%%%%%%%%%%%%%%%%%%%%%%%%
\begin{theorem}\label{TparaIntroduction}
Let $\Omega\subset\mathbb{R}^d$ be an arbitrary bounded smooth domain for $d=1,2$ and $B>0$, $T\ge 0$. Let $\lambda>0$ and $\kappa\in (0,1)$. Let $u^0\in H^{2}(\Omega)$ be such that $u^0\ge -1+\kappa$ in $\Omega$ and $u^0$ satisfies the boundary conditions \eqref{hyper2p}. Then the following hold:
\begin{itemize}
\item[(i)] There are $\tau_m>0$ and a unique  maximal solution $u$ to \eqref{hyper1p}-\eqref{hyper3p} with regularity
$$
u\in C([0,\tau_m),H^{2}(\Omega))\cap C((0,\tau_m),H^4(\Omega))\cap C^1((0,\tau_m),L_2(\Omega))\ .
$$

\item[(ii)] If $\tau_m<\infty$, then $$\liminf_{t\rightarrow\tau_m}\big(\min_{\bar\Omega} u(t)\big)=-1\ .$$%$(1+u)^{-2}\not\in L_\infty (0,\tau_m; L_2(\Omega))$.

\item[(iii)] There is $\lambda_1(\kappa)>0$ such that $\tau_m=\infty$ provided that $\lambda\le \lambda_1(\kappa)$ and $\|u^0\|_{H^{2}}\le \kappa^{-1}$. In this case, $u\in L_\infty(0,\infty;H^{2}(\Omega))$ and 
$$
\inf_{(0,\infty)\times \Omega} u>-1\ .
$$ 

\item[(iv)] If $\Omega=\unb$ and $u^0$ is radially symmetric, then so is $u(t)$ for each $t\in [0,\tau_m)$.
\end{itemize}
\end{theorem}
%%%%%%%%%%%%%%%%%%%%%%%%%%%%%%%%%%%%%%%%%%%%%%%%%%%%%%%%%%%%%%%%%%%%

The proof of Theorem~\ref{TparaIntroduction} follows the same lines as that of Theorem~\ref{ThyperIntroduction} and is to be found in Section~\ref{sec32a}. 

\bigskip

On physical grounds it is expected that solutions to the dynamic problems \eqref{hyper1}-\eqref{hyper3} or \eqref{hyper1p}-\eqref{hyper3p} touch down (i.e. $u=-1$) in finite time and thus do not exist globally if the voltage value $\lambda$ exceeds the critical pull-in voltage above which no stationary solution exists. This is true for the second-order parabolic case, see \cite{EGG10} for instance, but seems to be an open problem both for the hyperbolic equation \eqref{hyper1}-\eqref{hyper3} as well as for the parabolic equation \eqref{hyper1p}-\eqref{hyper3p}. Actually, even the weaker result of the occurrence of touchdown in finite time  for large values of $\lambda$ has not yet been proven apparently, though observed numerically in~\cite{LL12} for \eqref{hyper1p}-\eqref{hyper3p} with $T=0$  and shown in \cite{KLNT11} in the absence of bending ($B=0$). The next result is a step in that direction when $\Omega$ is the unit ball $\unb$ of $\mathbb{R}^d$.

%%%%%%%%%%%%%%%%%%%%%%%%%%%%%%%%%%%%%%%%%%%%%%%%%%%%%%%%%%%%%%%%%%%%
\begin{proposition}\label{asterix}
Let $\Omega=\unb$ with $d\in\{ 1, 2\}$, $\lambda>0$, $B>0$, $T\ge 0$, and let $(u^0,u^1)\in H^4(\unb)\times H^2(\unb)$ be such that $u^0>-1$ in $\unb$ and both $u^0$ and $u^1$ satisfy the boundary conditions \eqref{hyper2}. Let $u$ be the maximal solution to either \eqref{hyper1}-\eqref{hyper3} with initial condition $(u^0,u^1)$ enjoying the properties listed in Theorem~\ref{ThyperIntroduction} or \eqref{hyper1p}-\eqref{hyper3p} with initial condition $u^0$ enjoying the properties listed in Theorem~\ref{TparaIntroduction}. Let $\tau_m$ be its maximal existence time. If $\lambda$ is sufficiently large (see \eqref{aq0} below for a quantitative lower bound), then $\tau_m<\infty$.
\end{proposition}
%%%%%%%%%%%%%%%%%%%%%%%%%%%%%%%%%%%%%%%%%%%%%%%%%%%%%%%%%%%%%%%%%%%%

It is worth pointing out that the outcome of Proposition~\ref{asterix} complies with the numerical simulations of \eqref{hyper1p}-\eqref{hyper3p} performed in \cite{LL12} in $\unb$ and showing the occurrence of finite time touchdown. The proof of Proposition~\ref{asterix} is given in Section~\ref{sec41} and relies on the eigenfunction method.\\

Owing to the study carried out in Section~\ref{sec2}, we are able to refine this result in the radially symmetric setting and show that the touchdown behavior indeed starts exactly above the threshold value $\lambda_*$ defined in Theorem~\ref{TSSIntroduction}.

%%%%%%%%%%%%%%%%%%%%%%%%%%%%%%%%%%%%%%%%%%%%%%%%%%%%%%%%%%%%%%%%%%%%
\begin{proposition}\label{Majestix}
Assume $\Omega=\unb$ with $d\in\{ 1, 2\}$ and let $(u^0,u^1)$ be radially symmetric initial conditions satisfying the requirements of Theorem~\ref{ThyperIntroduction} if $\gamma>0$ or Theorem~\ref{TparaIntroduction} if $\gamma=0$. Then, if $\lambda>\lambda_*$, the corresponding maximal solution to \eqref{hyper1}-\eqref{hyper3} or \eqref{hyper1p}-\eqref{hyper3p} on $[0,\tau_m)$ does not exist globally, that is, $\tau_m<\infty$.
\end{proposition}
%%%%%%%%%%%%%%%%%%%%%%%%%%%%%%%%%%%%%%%%%%%%%%%%%%%%%%%%%%%%%%%%%%%%

The proof of Proposition~\ref{Majestix} is performed in Section~\ref{sec42} and also relies on the eigenfunction method, but with a more accurate choice than in the proof of Proposition~\ref{asterix} as already noticed in \cite{Gu10}.\\

Let us conclude the introduction with some remarks on the qualitative behavior of solutions to the evolution problem in the ball $\unb$. Proposition~\ref{asterix} and Proposition~\ref{Majestix} show the occurrence of a finite time singularity,
but do not provide information about the precise behavior near touchdown time. According to the numerical simulations performed in \cite{LL12}, the fourth-order term has a strong influence on the way solutions touch down in finite time as this might take place on a circle (for $d=2$). This markedly contrasts the second-order case, where touchdown occurs only at the single point $x=0$, see \cite[Theorem 8.3.4]{EGG10}.

When a solution to \eqref{hyper1}-\eqref{hyper3} does not touch down in finite time, then it exists globally in time and might even be bounded away from $-1$ as well as be bounded in $H^2$ according to Theorem~\ref{ThyperIntroduction} (if $\gamma>0$) and Theorem~\ref{TparaIntroduction} (if $\gamma=0$). A natural next step to understand its dynamics is to investigate its large time behavior. While this seems to be an open problem for a general domain~$\Omega$, the analysis performed in Section~\ref{sec2} for $\Omega=\unb$ in the radially symmetric setting paves the way for a better understanding of this issue. On the one hand, Proposition~\ref{thb5} below entails that one may apply the principle of linearized stability to show that $U(s)$ is locally asymptotically stable when $s\in [0,s_*)$.
On the other hand, it might be possible to use the \L ojasiewicz-Simon inequality as in \cite{Gu10} to establish convergence to a single steady-state.

%%%%%%%%%%%%%%%%%%%%%%%%%%%%%%%%%%%%%%%%%%%%%%%%%%%%%%%%%%%%%%%%%%%%
%%%%%%%%%%%%%%%%%%%%%%%%%%%%%%%%%%%%%%%%%%%%%%%%%%%%%%%%%%%%%%%%%%%%
\section{Radially symmetric stationary solutions}\label{sec2}
%%%%%%%%%%%%%%%%%%%%%%%%%%%%%%%%%%%%%%%%%%%%%%%%%%%%%%%%%%%%%%%%%%%%
%%%%%%%%%%%%%%%%%%%%%%%%%%%%%%%%%%%%%%%%%%%%%%%%%%%%%%%%%%%%%%%%%%%%

In the following, if $S(\Omega)$ is a space of functions defined on $\Omega$, we write $S_D(\Omega)$ for the subspace of $S(\Omega)$ consisting of  functions $u$ satisfying the Dirichlet boundary conditions \eqref{hyper2s}, if meaningful. If $\Omega$ is the unit ball $\unb$ of $\RR^d$, then  $S_r(\unb)$ stands for the subspace of $S(\unb)$ consisting of radially symmetric functions. Clearly, $S_{D,r}(\unb):=S_D(\unb)\cap S_r(\unb)$.  If $S(\Omega)$ is a normed vector space, then $\|\cdot\|_S$ stands for its norm. For $p\in [1,\infty]$ we denote the norm of $L_p(\Omega)$ simply by $\|\cdot\|_p$.\\
 
 %%%%%%%%%%%%%%%%%%%%%%%%%%%%

Recall that the stationary solutions of \eqref{hyper1}-\eqref{hyper2} satisfy
\begin{align}
B \Delta^2 u - T \Delta u & =  - \lambda g(u)\!\!\!\!\!\!\!\! &&\text{ in }\;\; \unb\,, \label{stat1} \\
u = \partial_\nu u & =  0 &&\text{ on }\;\; \uns\,,     \label{stat2}
\end{align}
where $g(\xi):=(1+\xi)^{-2}$ for $\xi>-1$.

%%%%%%%%%%%%%%%%%%%%%%%%%%%%%%%%%%%%%%%%%%%%%%%%%%%%%%%%%%%%%%%%%%%%
\begin{definition}\label{def}
A radially symmetric classical solution $u$ (with parameter $\lambda$) of the boundary value problem \eqref{stat1}-\eqref{stat2} is a radially symmetric function $u\in C_r^4(\unb)\cap C_r^2(\cunb)$ satisfying $u>-1$ in $\cunb$ and solving \eqref{stat1}-\eqref{stat2} in the classical sense.

We denote the set of all radially symmetric classical solutions with parameter $\lambda$ to the boundary value problem \eqref{stat1}-\eqref{stat2} by $\mathcal{S}_r^\lambda$. 
\end{definition}
%%%%%%%%%%%%%%%%%%%%%%%%%%%%%%%%%%%%%%%%%%%%%%%%%%%%%%%%%%%%%%%%%%%%

Similarly, a radially symmetric function $u\in C_r^4(\unb)\cap C_r^2(\cunb)$ satisfying $u>-1$ in $\cunb$ is a classical subsolution of \eqref{stat1}-\eqref{stat2} (with parameter $\lambda$), if it satisfies $B \Delta^2 u - T \Delta u \le - \lambda g(u)$ in $\unb$ and \eqref{stat2} on $\uns$.

\bigskip

We introduce the operator 
$$
Au:=B\Delta^2 u-T\Delta u\ ,\quad u\in H_{D,r}^4(\unb)\ ,
$$
and recall the following well-known properties:

%%%%%%%%%%%%%%%%%%%%%%%%%%%%%%%%%%%%%%%%%%%%%%%%%%%%%%%%%%%%%%%%%%%%
\begin{lemma}\label{L0}
 $A\in\ml(H_{D,r}^4(\unb),L_{2,r}(\unb))$ is invertible with $$A^{-1}\in\ml(L_{2,r}(\unb),H_{D,r}^4(\unb))\cap \ml(C_{D,r}^\alpha(\cunb),C_{D,r}^{4+\alpha}(\cunb))$$ for each $\alpha\in (0,1)$. Moreover, there are $m_1>0$ and $\phi_1\in C_{D,r}^{4}(\cunb)$ with $\phi_1>0$ in $\unb$,  $\|\phi_1\|_1=1$, and
$A\phi_1 =m_1 \phi_1$.
\end{lemma}
%%%%%%%%%%%%%%%%%%%%%%%%%%%%%%%%%%%%%%%%%%%%%%%%%%%%%%%%%%%%%%%%%%%%

\begin{proof}
The invertibility of $A$ and the regularity properties of $A^{-1}$ are consequences of \cite[Theorem~2.15, Theorem~2.19, Theorem~2.20]{GGS10}. That there is a positive eigenvalue with a positive eigenfunction follows from \cite[Theorem 4.7]{LWxx}.
\end{proof}

\medskip

We further define
\begin{equation}
\lambda_* := \sup{ \left\{ \lambda>0\, :\, \mathcal{S}_r^\lambda \;\text{ is non-empty} \right\} } \in [0,\infty]\ , \label{stat3}
\end{equation}
and first derive some elementary properties of $\mathcal{S}_r^\lambda$.

%%%%%%%%%%%%%%%%%%%%%%%%%%%%%%%%%%%%%%%%%%%%%%%%%%%%%%%%%%%%%%%%%%%%
\begin{lemma}\label{L1a}
The following hold:
\begin{itemize}
\item[(i)] $\mathcal{S}_r^0=\{0\}$ and $\mathcal{S}_r^\lambda\subset C_{D,r}^{4}(\cunb)$ for $\lambda>0$;
\item[(ii)] if $\lambda>0$ and $u\in \mathcal{S}_r^\lambda$, then $-1<u\le 0$ in $\cunb$;
\item[(iii)] the threshold value $\lambda_*$ defined in \eqref{stat3} is finite.
\end{itemize}
\end{lemma}
%%%%%%%%%%%%%%%%%%%%%%%%%%%%%%%%%%%%%%%%%%%%%%%%%%%%%%%%%%%%%%%%%%%%

\begin{proof}
The first statement of (i) readily follows from Lemma~\ref{L0}. If $u\in\mathcal{S}_r^\lambda$,
then $g(u)$ belongs to $C^{2 }(\cunb)$ since $g$ is smooth in $(-1,\infty)$ and $u>-1$ in $\cunb$. Thus $u\in C_{D,r}^{4}(\cunb)$ by Lemma~\ref{L0}. Moreover, $-1<u\le 0$ in $\cunb$ by \cite[Theorem 1.4]{LWxx} since $\lambda g(u)\ge 0$ (see also Lemma~\ref{leb1} below). Consequently, testing \eqref{stat1}-\eqref{stat2} by $\phi_1>0$ introduced in Lemma~\ref{L0} yields
$$
-m_1\int_{\unb}\phi_1\,\rd x \le m_1\int_{\unb}\phi_1\, u\,\rd x =-\lambda\int_{\unb}\phi_1\,g(u)\,\rd x\le -\lambda\int_{\unb}\phi_1\,\rd x\ ,
$$
whence $\lambda\le m_1$. Therefore, $\lambda_*\le m_1<\infty$.
\end{proof}

In fact, one can show that  $\lambda_*< m_1$. Indeed, assume $\lambda_*= m_1$ for contradiction  so that there are sequences $\lambda_n\rightarrow m_1$ and $u_n\in\mathcal{S}_r^{\lambda_n}$. Then, the above computation actually yields
$$
m_1-\lambda_n\ge \lambda_n\int_{\unb} \phi_1\, (g(u_n)-1)\,\rd x\ge 0\ .
$$
Since also
$$
\int_{\unb} \phi_1\, (g(u_n)-1)\,\rd x = \int_{\unb} \phi_1\, \frac{\vert u_n\vert (u_n+2)}{(1+u_n)^2}\,\rd x \ge \int_{\unb} \phi_1 \vert u_n\vert\,\rd x\ ,
$$
we conclude that
$$
\lim_{n\rightarrow \infty}  \lambda_n\int_{\unb} \phi_1\, g(u_n)\,\mathrm{d} x =m_1
\quad\text{and}\quad\lim_{n\rightarrow \infty}   \int_{\unb} \phi_1\, \vert u_n\vert\,\mathrm{d} x=0\ .
$$
From $u_n\in\mathcal{S}_r^{\lambda_n}$ we obtain
$$
 m_1\int_{\unb}\phi_1\, u_n\,\rd x =-\lambda_n\int_{\unb}\phi_1\,g(u_n)\,\rd x\ ,
$$
and letting $n\rightarrow\infty$ and using the previous limits imply that $m_1=0$, which contradicts Lemma~\ref{L0}.

\begin{remark}\label{RRR1}
Observe that the computation in the proof of Lemma~\ref{L1a} excludes also the existence of non-radially symmetric solutions to \eqref{hyper1s}-\eqref{hyper2s} for $\lambda>m_1$. An interesting question is whether there are non-radially symmetric solutions  for $\lambda\in (\lambda_*,m_1)$. This is not the case when $T=0$ as it is shown in \cite{BGW08} that all solutions to \eqref{hyper1s}-\eqref{hyper2s} are radially symmetric.
\end{remark}

%%%%%%%%%%%%%%%%%%%%%%%%%%%%%%%%%%%%%%%%%%%%%%%%%%%%%%%%%%%%%%%%%%%%
%%%%%%%%%%%%%%%%%%%%%%%%%%%%%%%%%%%%%%%%%%%%%%%%%%%%%%%%%%%%%%%%%%%%
\subsection{A continuous curve of stationary solutions}\label{sec21}
%%%%%%%%%%%%%%%%%%%%%%%%%%%%%%%%%%%%%%%%%%%%%%%%%%%%%%%%%%%%%%%%%%%%
%%%%%%%%%%%%%%%%%%%%%%%%%%%%%%%%%%%%%%%%%%%%%%%%%%%%%%%%%%%%%%%%%%%%

In this subsection we invoke the global bifurcation theory of  \cite[Section~2.1]{BDT00} for real analytic functions to establish the existence of a global curve of radially symmetric stationary solutions. This tool has also been used in \cite[Section~6.2]{EGG10} for the second-order case (that is, $B=0$). 

Since 
\begin{equation}\label{O}
\mathcal{O}:=\{u\in C_{D,r}^{1}(\cunb)\,:\, \vert u\vert <1\text{ in } \cunb\}
\end{equation}
is open in $C_{D,r}^{1}(\cunb)$, the mapping
\begin{equation}\label{F}
F:\RR\times\mathcal{O}\rightarrow C_{D,r}^{1}(\cunb)\ ,\quad (\lambda, u)\mapsto u+\lambda A^{-1} g(u)\ ,
\end{equation}
is well-defined according to Lemma~\ref{L0} and real analytic. Observe that $u\in\mathcal{S}_r^\lambda$ if $(\lambda,u)\in \RR\times\mathcal{O}$ with $F(\lambda,u)=0$, the bound $\vert u\vert<1$ following from Lemma~\ref{L1a}. Clearly, $F(0,0)=0$ and the partial Fr\'echet derivative $F_u(0,0)$ equals the identity in $C_{D,r}^{1}(\cunb)$. Thus, by the implicit function theorem, the zeros of $F$ near $(0,0)$ are given by a real analytic curve $(\lambda, V(\lambda))$ with $V(0)=0$. Moreover, there exists $\lambda_0\in (0,\infty]$, which is maximal with respect to the existence of a real analytic function $V:[0,\lambda_0)\rightarrow C_{D,r}^{1}(\cunb)$ for which $F(\lambda, V(\lambda))=0$ and $F_u(\lambda, V(\lambda))\in \ml(C_{D,r}^{1}(\cunb))$ is boundedly invertible for each $\lambda\in [0,\lambda_0)$. Consequently, the set
\begin{equation*}
\begin{split}
S:=\{(\lambda, u)\in (0,\infty)\times \mathcal{O}\,:\, F(\lambda,u)=0 \text{ and } &F_u(\lambda,u)\in \ml(C_{D,r}^{1}(\cunb))\\
& \text{ is boundedly invertible}\}
\end{split}
\end{equation*}
is non-empty as it contains the maximal arc-connected subset
\begin{equation}\label{A0}
\mathcal{A}_0:=\{(\lambda, V(\lambda))\,:\, \lambda\in (0,\lambda_0)\}\ .
\end{equation}
Note that $\lambda_0$ and $V$ are unique and necessarily $\lambda_0$ is finite since it belongs to $(0,\lambda_*]$. We have thus verified assumption (C1) from
\cite[Section 2.1]{BDT00}. For (C2) therein we may argue as in \cite[p.128]{EGG10} that this assumption  merely serves to show in the proof of \cite[Theorem 2.3 (iii)]{BDT00} that $S$ is open in its closure $\bar{S}$  and can thus be replaced by the stronger one that $(0,\infty)\times\mathcal{O}$ is open in $\RR\times C_{D,r}^{1}(\cunb)$. Then, since  $H_{D,r}^4(\unb)$ embeds compactly in   $C_{D,r}^{1}(\cunb)$, we may regard the operator $\lambda A^{-1} g'(u)\in\ml(C_{D,r}^{1}(\cunb),H_{D,r}^4(\unb))$ as a compact operator in $C_{D,r}^{1}(\cunb)$ for each $(\lambda,u)\in\RR\times\mathcal{O}$. Hence, by \cite[Theorem 4.25]{RudinFA}, the partial Fr\'echet derivative
$$
F_u(\lambda,u)=1+\lambda A^{-1} g'(u)\ ,\quad (\lambda,u)\in\RR\times\mathcal{O}\ ,
$$
is a Fredholm operator of index 0. The remark in \cite[p. 246]{BDT00} now entails that (C3)-(C5) therein hold. Next, we introduce the function
$$
\nu:(0,\infty)\times\mathcal{O}\rightarrow [0,\infty)\ ,\quad (\lambda,u)\mapsto \frac{1}{\min_{\cunb} \{1+u\}}\ .
$$
To verify (C6) from \cite{BDT00} consider a sequence $(\lambda_n,u_n)_{n\in\NN}$ in $(0,\infty)\times\mathcal{O}$ with $F(\lambda_n,u_n)=0$ and $\nu(\lambda_n,u_n)\le c<\infty$ for each $n\in \NN$. Then, by Lemma~\ref{L0}, each $u_n$ belongs to $C_{D,r}^4(\cunb)$ with $u_n\ge -1+c^{-1}$ in $\cunb$ and satisfies
$$
B \Delta^2 u_n - T \Delta u_n = - \lambda_n g(u_n) \;\;\text{ in }\;\; \unb\,.
$$
The above uniform lower bound on $u_n$ and the finiteness of $\lambda_*$ established in
Lemma~\ref{L1a} now imply that the sequence $(\lambda_ng(u_n))_{n\in\NN}$ is bounded in $L_{\infty}(\unb)$. Thus, $(\lambda_n,u_n)_{n\in\NN}$ is bounded in $(0,\lambda_*]\times H_{D,r}^4(\unb)$ and so $(\lambda_n,u_n)_{n\in\NN}$ has a converging subsequence in $[0,\lambda_*]\times C_{D,r}^{1}(\cunb)$. Hence (C6) in~\cite{BDT00} holds true. Setting $\bar{\lambda}:=0$, we clearly have $(\bar{\lambda},0)\not\in (0,\infty)\times\mathcal{O}$ and $(\bar{\lambda},0)$ is in the closure of $\mathcal{A}_0$ defined in \eqref{A0}, whence (C7) in \cite{BDT00}. Finally, suppose that $(\lambda_n,u_n)_{n\in\NN}$ is a sequence in $(0,\infty)\times\mathcal{O}$ with $F(\lambda_n,u_n)=0$ and $\nu(\lambda_n,u_n)\le c<\infty$ for each $n\in \NN$, which converges in $\RR\times C_{D,r}^{1}(\cunb)$ towards $(\lambda,u)\not\in (0,\infty)\times \mathcal{O}$. Then $-1+c^{-1}\le u_n\le 0$ in $\unb$ and $0\le \lambda_n\le \lambda_*$ for each $n\in\NN$ by  Lemma~\ref{L1a}, which entails that $-1+c^{-1}\le u\le 0$, whence $(\lambda,u)\in [0,\lambda_*]\times \mathcal{O}$ and $F(\lambda,u)=0$. Since $(\lambda,u)\not\in (0,\infty)\times \mathcal{O}$, this is only possible if $(\lambda,u)=(0,0)$. The implicit function theorem guarantees that $(\lambda_n,u_n)\in\mathcal{A}_0$ for $n$ large enough. This yields (C8) in \cite{BDT00}, and therefore, we are in a position to apply \cite[Theorem 2.4]{BDT00} and obtain:

%%%%%%%%%%%%%%%%%%%%%%%%%%%%%%%%%%%%%%%%%%%%%%%%%%%%%%%%%%%%%%%%%%%%
\begin{theorem}\label{T19}
There is a continuous function $(\Lambda,U):(0,\infty)\rightarrow (0,\infty)\times C_{D,r}^{1 }(\cunb)$ with the following  properties:
\begin{itemize}
	\item[(i)] $U(s)\in\mathcal{S}_r^{\Lambda(s)}$ for each $s\in (0,\infty)$;
	\item[(ii)] $(\Lambda,U)((0,1))\subset \mathcal{A}_0$ and $\lim_{s\rightarrow 0} (\Lambda(s),U(s))=(0,0)$;
	\item[(iii)] $(\Lambda,U)$ is injective on $(\Lambda,U)^{-1}(S)$ and $$\lim_{s\rightarrow\infty}\big(\min_{\cunb} U(s)\big)=-1\ ;$$
	\item[(iv)]  at all points $s\in (\Lambda,U)^{-1}(S)$, $(\Lambda,U)$ is real analytic with $\Lambda'(s)\not= 0$. 
\end{itemize}
\end{theorem}
%%%%%%%%%%%%%%%%%%%%%%%%%%%%%%%%%%%%%%%%%%%%%%%%%%%%%%%%%%%%%%%%%%%%

Actually, more precise information is given in \cite[Theorem 2.4]{BDT00} about the curve 
\begin{equation}\label{A}
\mathcal{A}:=\{(\Lambda(s),U(s))\,:\, s\in (0,\infty)\}\ ,
\end{equation}
traced out by the function $(\Lambda,U)$, in particular, that it is piecewise analytic:

%%%%%%%%%%%%%%%%%%%%%%%%%%%%%%%%%%%%%%%%%%%%%%%%%%%%%%%%%%%%%%%%%%%%
\begin{remark}\label{RR1}
The set $(\Lambda,U)^{-1}(\bar{S}\setminus S)\subset (0,\infty)$ consists of isolated values and locally near each point $s_0\in(\Lambda,U)^{-1}(\bar{S}\setminus S)$, there is a re-parametrization $\zeta$ of the parameter $s$ such that $(\Lambda,U)\circ \zeta$ is real analytic with derivative vanishing possibly only at 0.
\end{remark}
%%%%%%%%%%%%%%%%%%%%%%%%%%%%%%%%%%%%%%%%%%%%%%%%%%%%%%%%%%%%%%%%%%%%

Before analyzing further the curve $\mathcal{A}$ and in particular showing that it ``globally'' extends $\mathcal{A}_0$ defined in \eqref{A0} (note that at this point, the curves $\mathcal{A}_0$ and $\mathcal{A}$ could still coincide), we first derive general properties of  solutions to \eqref{hyper1s}-\eqref{hyper2s} in the next subsection .

%%%%%%%%%%%%%%%%%%%%%%%%%%%%%%%%%%%%%%%%%%%%%%%%%%%%%%%%%%%%%%%%%%%%
%%%%%%%%%%%%%%%%%%%%%%%%%%%%%%%%%%%%%%%%%%%%%%%%%%%%%%%%%%%%%%%%%%%%
\subsection{Properties of stationary solutions}\label{sec22}
%%%%%%%%%%%%%%%%%%%%%%%%%%%%%%%%%%%%%%%%%%%%%%%%%%%%%%%%%%%%%%%%%%%%
%%%%%%%%%%%%%%%%%%%%%%%%%%%%%%%%%%%%%%%%%%%%%%%%%%%%%%%%%%%%%%%%%%%%

We first recall the following sign-preserving property of the operator $B\Delta^2 - T\Delta$ with homogeneous clamped boundary conditions in $\unb$ with radial symmetry established in \cite{LWxx}.

%%%%%%%%%%%%%%%%%%%%%%%%%%%%%%%%%%%%%%%%%%%%%%%%%%%%%%%%%%%%%%%%%%%%
\begin{lemma}\label{leb1} 
Consider $f\in C_r(\bar{\mathbb{B}}_1)$ and $w\in C_r^4(\unb)\cap C_r^2(\cunb)$ such that $w$ is a classical solution to
\begin{eqnarray*}
B \Delta^2 w - T \Delta w & = & f \;\;\text{ in }\;\; \unb\,, \\
w = \partial_\nu w & = & 0 \;\;\text{ on }\;\; \uns\,.
\end{eqnarray*}
Then, if $f\le 0$ in $\unb$,
\begin{equation}
\text{ either }\;\; w \equiv 0 \;\;\text{ or }\;\; w < 0 \;\;\text{ in }\;\; \unb\ . \label{b3}
\end{equation}
Furthermore, 
\begin{equation}
\min_{\cunb} w = w(0)\ , \label{b4}
\end{equation}
and there is $r_0\in (0,1)$ such that
\begin{equation}
\Delta w < 0 \;\;\text{ in }\;\; \cunb\setminus\bar{\mathbb{B}}_{r_0} \;\;\text{ and }\;\; \Delta w > 0 \;\;\text{ in }\;\; \mathbb{B}_{r_0}\setminus\{0\}\ .  \label{b5}
\end{equation}
Finally, the profile $\mathfrak{w}$ of $w$ defined by $\mathfrak{w}(\vert x\vert) = w(x)$ for $x\in\cunb$ is a non-decreasing function on $[0,1]$. 
\end{lemma}
%%%%%%%%%%%%%%%%%%%%%%%%%%%%%%%%%%%%%%%%%%%%%%%%%%%%%%%%%%%%%%%%%%%%

\begin{proof}
The first statement \eqref{b3} of Lemma~\ref{leb1} readily follows from \cite[Theorem~1.4]{LWxx}. Furthermore, the proof of that result reveals that \eqref{b5} is true. We next deduce from \eqref{b5} that $\partial_r \left( r^{d-1} \partial_r \mathfrak{w}(r) \right)<0$ for $r\in (r_0,1]$ and $\partial_r \left( r^{d-1} \partial_r \mathfrak{w}(r) \right)>0$ for $r\in (0,r_0)$. Since $\partial_r \mathfrak{w}(0)=\partial_r \mathfrak{w}(1)=0$ due to the radial symmetry of $w$, its regularity, and its boundary conditions, we conclude that $\partial_r \mathfrak{w}(r)\ge 0$ for $r\in [0,1]$. Then $\mathfrak{w}$ is a non-decreasing function in $[0,1]$ and attains its minimum at $r=0$.
\end{proof}

%%%%%%%%%%%%%%%%%%%%%%%%%%%%%%%%%%%%%%%%%%%%%%%%%%%%%%%%%%%%%%%%%%%%
\begin{lemma} \label{leb3} 
Define the scalar product $\langle \cdot ,\cdot\rangle$ on $H_{D,r}^2(\unb)$ by
$$
\langle v ,w\rangle :=\int_{\unb} \left[ B \Delta v(x) \Delta w(x) + T \nabla v(x)\cdot \nabla w(x) \right]\, \mathrm{d} x\ ,\quad v, w\in H_{D,r}^2(\unb)\ .
$$
Let $\mathcal{K}:=\left\{ v\in H_{D,r}^2(\unb)\,:\, v\ge 0 \right\}$ be the positive cone of $H_{D,r}^2(\unb)$ and define its polar cone by $$
\mathcal{K}^\circ:=\{w\in H_{D,r}^2(\unb)\,:\, \langle v ,w\rangle \le 0\ \text{ for all } v\in \mathcal{K}\}\ .
$$
Then, given $v\in H_{D,r}^2(\unb)$, there is a unique couple $(v_1,v_2)\in\mathcal{K}\times \mathcal{K}^\circ$ such that \mbox{$\langle v_1, v_2 \rangle = 0$} and $v=v_1+v_2$. In addition, $v_2\le 0$ a.e. in $\unb$.
\end{lemma}
%%%%%%%%%%%%%%%%%%%%%%%%%%%%%%%%%%%%%%%%%%%%%%%%%%%%%%%%%%%%%%%%%%%%

\begin{proof}
The fact that any $v\in H_{D,r}^2(\unb)$ can be written in a unique way  as a sum $v=v_1+v_2$ with $\langle v_1,v_2\rangle =0$ and $(v_1,v_2)\in \mathcal{K}\times\mathcal{K}^\circ$ is a well-known result due to Moreau~\cite{Mo62}. The non-negativity property of $v_2$ actually follows from the sign-preserving property stated in Lemma~\ref{leb1} and can be proved as \cite[Proposition~4.5]{LWxx}, where only the one-dimensional case was handled.
\end{proof}

%%%%%%%%%%%%%%%%%%%%%%%%%%%%%%%%%%%%%%%%%%%%5

As in \cite{EGG10} the linear stability of stationary solutions is an important tool in the detailed analysis to follow. 
For $u\in \mathcal{S}_r^\lambda$, it is measured by 
\begin{equation}
\mu_1(u) := \inf{\left\{ \int_{\unb} \left( B |\Delta v|^2 + T |\nabla v|^2 + \lambda g'(u) v^2 \right)\,\rd x \ :\ v\in H_{D,r}^2(\unb)\ , \;\; \|v\|_2=1 \right\}}\ , \label{b6}
\end{equation}
which turns out to be  a simple eigenvalue of the linearization of \eqref{stat1} when non-negative as shown in the following lemma.

%%%%%%%%%%%%%%%%%%%%%%%%%%%%%%%%%%%%%%%%%%%%%%%%%%%%%%%%%%%%%%%%%%%%
\begin{lemma}\label{leb3c}
Consider $\lambda\in [0,\lambda_*]$ and $u\in\mathcal{S}_r^\lambda$ such that $\mu_1(u)\ge 0$. Then the following hold:
\begin{itemize}
\item[(i)] $\mu_1(u)$ is a simple eigenvalue of the operator $A + \lambda g'(u)\in \ml(H_{D,r}^4(\unb),L_{2,r}(\unb)) $ with a positive eigenfunction in $C_{D,r}^4(\cunb)$;
\item[(ii)] $\mu_1(u)>0$ if and only if $F_u(\lambda,u)=1+\lambda A^{-1}g'(u)\in\mathcal{L}(C_{D,r}^1(\cunb))$ is boundedly invertible.
\end{itemize}
\end{lemma}
%%%%%%%%%%%%%%%%%%%%%%%%%%%%%%%%%%%%%%%%%%%%%%%%%%%%%%%%%%%%%%%%%%%%

\begin{proof}
(i) A classical compactness argument along with the weak lower semicontinuity of the scalar product $\langle \cdot , \cdot \rangle$ in $H_{D,r}^2(\unb)$ defined in Lemma~\ref{leb3} guarantee the existence of a minimizer $\phi$ to \eqref{b6} in $H_{D,r}^2(\unb)$ satisfying $\|\phi\|_2=1$. Then $\phi\in H_{D,r}^4(\unb)$  is a solution to the corresponding Euler-Lagrange equation
\begin{equation}
B \Delta^2 \phi - T \Delta \phi + \left( \lambda g'(u) - \mu_1(u) \right) \phi = 0 \;\;\text{ in }\;\; \unb\ , \quad \phi=\partial_\nu \phi = 0 \;\;\text{ on }\;\; \uns\ . \label{b200}
\end{equation}
Now, let $\phi\in H_{D,r}^4(\unb)$ be any solution to the boundary value problem \eqref{b200}. According to Lemma~\ref{leb3}, there is a unique couple $(\phi_1,\phi_2)\in \mathcal{K}\times\mathcal{K}^\circ$ such that $\phi=\phi_1+\phi_2$, $\langle \phi_1, \phi_2 \rangle = 0$, and $\phi_2\le 0$ a.e. in $\unb$. We deduce from the definition \eqref{b6} of $\mu_1(u)$, the orthogonality properties of $(\phi_1,\phi_2)$, and \eqref{b200} that
\begin{align*}
\mu_1(u) \|\phi_1-\phi_2\|_2^2 & \le \langle \phi_1 - \phi_2 , \phi_1 - \phi_2 \rangle + \lambda \int_{\unb} g'(u) (\phi_1-\phi_2)^2\ \mathrm{d}x \\
& \le \langle \phi_1 + \phi_2 , \phi_1 + \phi_2 \rangle + \lambda \int_{\unb} g'(u) (\phi_1-\phi_2)^2\ \mathrm{d}x \\
& \le \lambda \int_{\unb} g'(u) \left[ (\phi_1-\phi_2)^2 - (\phi_1+\phi_2)^2 \right]\ \mathrm{d}x + \mu_1(u) \|\phi_1+\phi_2\|_2^2\ ,
\end{align*}
whence
$$
0 \le - 4\lambda \int_{\unb} g'(u) \phi_1 \phi_2\ \mathrm{d}x + 4\mu_1(u) \int_{\unb} \phi_1 \phi_2\ \mathrm{d}x\ .
$$
Both terms of the right-hand side of the above inequality being non-positive, we infer from the negativity of $g'$ that
\begin{equation}
\phi_1 \phi_2 = 0 \;\;\text{ a.e. in }\;\; \unb\,. \label{b201}
\end{equation}

Now, for $i=1,2$, it follows from the embedding of $H^2(\unb)$ in $C^\alpha(\cunb)$ for $\alpha\in (0,1)$ (recall that $d\in \{1,2\}$) that $\phi_i\in C_r^\alpha(\cunb)$ and, according to \cite[Theorem~2.19]{GGS10}, the boundary value problem 
\begin{equation}
B \Delta^2 \psi_i - T \Delta\psi_i = \left[ \mu_1(u) - \lambda g'(u) \right] \phi_i \;\;\text{ in }\;\; \unb\ , \qquad \psi_i = \partial_\nu \psi_i = 0 \;\;\text{ on }\;\; \uns\ , \label{b202}
\end{equation}
has a unique classical radially symmetric solution $\psi_i\in C_r^{4+\alpha}(\cunb)$. Since $\mu_1(u) - \lambda g'(u)>0$, $\phi_1\ge 0$ and $\phi_2\le 0$ in $\unb$, it follows from Lemma~\ref{leb1} that $\psi_1\ge 0 \ge \psi_2$ in $\unb$ with $\psi_1>0$ in $\unb$ if $\phi_1\not\equiv 0$ and $\psi_2<0$ in $\unb$ if $\phi_2\not\equiv 0$. In addition, due to \eqref{b200} and \eqref{b202}, 
$$
B \Delta^2 (\phi - \psi_1 - \psi_2 ) - T \Delta (\phi-\psi_1-\psi_2) = \left[ \mu_1(u) - \lambda g'(u) \right] (\phi - \phi_1 -\phi_2) = 0 \;\;\text{ in }\;\; \unb
$$
with $\phi-\psi_1-\psi_2\in H^4_{D,r}(\unb)$, whence $\phi=\psi_1+\psi_2$. Furthermore, $\psi_1$ clearly belongs to $\mathcal{K}$ while, for any $v\in\mathcal{K}$, we infer from \eqref{b202} that
$$
\langle \psi_2 , v \rangle = \int_{\unb} \left( \mu_1(u) - \lambda g'(u) \right) \phi_2 v\ \mathrm{d}x\le 0\ ,
$$
so that $\psi_2\in\mathcal{K}^\circ$. The uniqueness of Moreau's decomposition then warrants that $\psi_i=\phi_i$ for $i=1,2$. Therefore, if $\phi_1\not\equiv 0$ and $\phi_2\not\equiv 0$, we deduce from the above analysis that $\psi_1 \psi_2 < 0$ a.e. in $\unb$ and $\psi_1 \psi_2 = \phi_1 \phi_2 = 0$ a.e. in $\unb$, and a contradiction. Therefore, either $\phi_1\equiv 0$ or $\phi_2\equiv 0$, and we have shown that $\phi$ does not change sign in $\unb$. 

Consequently, any element of the kernel of the operator $A + \lambda g'(u)-\mu_1(u)$ in $H^4_{D,r}(\unb)$ does not change sign, which implies that the kernel's dimension is one by a classical argument. Indeed, assume for contradiction that there are two linearly independent
positive functions $\phi$ and $\psi$ in the kernel. Then $\phi-\alpha\psi$ with suitable $\alpha>0$ is a
sign-changing function in the kernel, which is impossible. Therefore, the kernel of $A + \lambda g'(u)$ is spanned by a positive function $\phi\in C_{D,r}^4(\cunb)$, the additional regularity stemming from Lemma~\ref{L0}.
 Finally, to show that $\mu_1(u)$ is a simple eigenvalue of $A + \lambda g'(u)$, consider $\Phi\in H_{D,r}^4(\unb)$ such that $A\Phi\in H_{D,r}^4(\unb)$ and $(A + \lambda g'(u)-\mu_1(u))^2\Phi =0$. Then, $(A + \lambda g'(u)-\mu_1(u))\Phi=\alpha \phi$ for some $\alpha\in\RR$. Multiplying this identity by $\phi$ and integrating over $\unb$ gives $\alpha\|\phi\|_2^2=0$, thus $\alpha=0$.
This yields assertion~(i).

(ii) Assume that  $F_u(\lambda,u)=1+\lambda A^{-1}g'(u)\in\mathcal{L}(C_{D,r}^1(\cunb))$ is not boundedly invertible. Then $-1$ is an eigenvalue of the compact operator $\lambda A^{-1}g'(u)\in\mathcal{L}(C_{D,r}^1(\cunb))$. Hence there is $\phi\in C_{D,r}^1(\cunb)$ with $\phi+\lambda A^{-1}g'(u)\phi=0$. Alternatively, $A\phi=-\lambda g'(u)\phi$ so that $\phi\in H_{D,r}^4(\unb)$ by Lemma~\ref{L0} and 
$$
\mu_1(u)\|\phi\|_2^2\le\langle \phi,\phi\rangle +\lambda\int_{\unb} g'(u)\phi^2\,\rd x=0\ ,
$$
which implies $\mu_1(u)\le 0$. Conversely, if $\mu_1(u)= 0$, then, arguing as in the proof of Lemma~\ref{leb3c}, we obtain a solution $\phi\in H_{D,r}^4(\unb)$ to
\begin{equation*}
B \Delta^2 \phi - T \Delta \phi + \lambda g'(u) \phi = 0 \;\;\text{ in }\;\; \unb\ , 
\end{equation*}
and thus $\phi+\lambda A^{-1}g'(u)\phi=0$.
\end{proof}

As in \cite[Chapter~11]{EGG10}, a key tool in the analysis is the following comparison lemma.

%%%%%%%%%%%%%%%%%%%%%%%%%%%%%%%%%%%%%%%%%%%%%%%%%%%%%%%%%%%%%%%%%%%%
\begin{lemma}\label{leb4} 
Consider $\lambda\in (0,\lambda_*]$ and $u\in \mathcal{S}_r^\lambda$ such that $\mu_1(u)\ge 0$. 
\begin{itemize}
\item[(i)] If $v\in C_{D,r}^4(\unb)\cap C_r^2(\cunb)$ is a classical subsolution to \eqref{stat1}-\eqref{stat2} with $v>-1$ in $\cunb$, then $v\le u$ in $\unb$.
\item[(ii)] Furthermore, $v=u$ if $\mu_1(u)=0$.
\end{itemize}
\end{lemma}
%%%%%%%%%%%%%%%%%%%%%%%%%%%%%%%%%%%%%%%%%%%%%%%%%%%%%%%%%%%%%%%%%%%%

\begin{proof}
(i) We proceed along the lines of the proof of \cite[Lemma~11.3.4]{EGG10}. By Lemma~\ref{leb3} there is a unique couple $(w_1,w_2)\in \mathcal{K}\times\mathcal{K}^\circ$ such that $v-u=w_1+w_2$, $\langle w_1, w_2 \rangle = 0$, and $w_2\le 0$ a.e. in $\unb$. Since 
$$
B \Delta^2 (v-u) - T \Delta (v-u) + \lambda (g(v)-g(u)) \le 0 \;\;\text{ in }\; \unb
$$
and $w_1\in\mathcal{K}$, we may multiply the above inequality by $w_1$ and integrate over $\unb$ to obtain
$$
\langle w_1 , v- u \rangle + \lambda \int_{\unb} (g(v)-g(u)) w_1\ \mathrm{d}x \le 0\ . 
$$
We next deduce from $\mu_1(u)\ge 0$ that
\begin{align*}
\langle w_1 , v- u \rangle = \langle w_1 , w_1 \rangle & \ge - \lambda \int_{\unb} g'(u) w_1^2\ \mathrm{d}x =  - \lambda \int_{\unb} g'(u) w_1 (v-u-w_2)\ \mathrm{d}x \\
& = - \lambda \int_{\unb} g'(u) w_1 (v-u)\ \mathrm{d}x + \lambda \int_{\unb} g'(u) w_1 w_2\, \mathrm{d}x\ .
\end{align*}
Combining the previous two inequalities gives
$$
\lambda \int_{\unb} (g(v)-g(u) - g'(u) (v-u)) w_1\ \mathrm{d}x + \lambda \int_{\unb} g'(u) w_1 w_2\, \mathrm{d}x \le 0\ .
$$
Owing to the convexity and the monotonicity of $g$ together with the sign properties of $w_1$ and $w_2$, the two terms on the left-hand side of the above inequality are non-negative. Therefore,
$$
(g(v)-g(u) - g'(u) (v-u)) w_1 = w_1 w_2 = 0 \;\;\text{ a.e. in}\;\; \unb
$$
and, in particular, 
$$
g(v)-g(u) - g'(u) (v-u) = w_2 = 0 \;\;\text{ a.e. in}\;\; \{x\in \unb\ :\ w_1(x)>0 \}\ .
$$
Since $g$ is strictly convex, this implies that $v-u=w_2=0$ a.e. in $\{x\in \unb\, :\, w_1(x)>0\}$. We have thus shown that $v-u=0$ a.e. in $\{x\in \unb\ :\ w_1(x)>0 \}$ and, since $v-u=w_2\le 0$ a.e. in $\{x\in \unb\ :\ w_1(x)=0 \}$, we conclude that $v-u\le 0$ a.e. in~$\unb$. 

\medskip

(ii) As in \cite[Lemma~11.3.4]{EGG10}, we define
$$
f(\vartheta) := \langle \vartheta v + (1-\vartheta) u , \phi \rangle + \lambda \int_{\unb} g(\vartheta v + (1- \vartheta) u) \phi\ \mathrm{d}x\ , \quad \vartheta\in [0,1]\ ,
$$
where $\phi$ is the unique positive eigenfunction of the linearized operator  $B \Delta^2 - T \Delta + \lambda g'(u) $ in $H^4_{D,r}(\unb)$ associated to the eigenvalue $\mu_1(u)=0$ satisfying $\|\phi\|_1=1$, see Lemma~\ref{leb3c}. Since $g$ is convex, $\phi>0$ in $\unb$, and $\vartheta v + (1-\vartheta) u$ satisfies
$$
B \Delta^2 (\vartheta v + (1-\vartheta) u) - T \Delta
(\vartheta v + (1-\vartheta) u) + \lambda \left( \vartheta g(v) + (1-\vartheta) g(u) \right) \le 0 \;\;\text{ in }\;\; \unb\ ,
$$
we conclude 
\begin{equation}
f(\vartheta) \le 0\ , \quad \vartheta\in [0,1]\,. \label{b7a}
\end{equation}
As 
$$
f'(\vartheta) = \langle v-u , \phi \rangle + \lambda \int_{\unb} g'(\vartheta v + (1-\vartheta) u) (v-u) \phi\ \mathrm{d}x
$$
and
$$
f''(\vartheta) = \lambda \int_{\unb} g''(\vartheta v + (1-\vartheta) u) (v-u)^2 \phi\ \mathrm{d}x\ ,
$$
the assumption $\mu_1(u)=0$ guarantees that $f'(0)=0$ while the convexity of $g$ and the positivity of $\phi$ imply that $f''(0)\ge 0$. In addition, recalling that $f(0)=0$, we deduce from \eqref{b7a} that $f''(0)\le 0$. Therefore, $f''(0)=0$ and the strict convexity of $g$ and the positivity of $\phi$ in $\unb$ entail $v=u$.
\end{proof}

In order to study more precisely the behavior of solutions to $\mathcal{S}_r^\lambda$ as the parameter $\lambda$ varies, we now derive several estimates.

%%%%%%%%%%%%%%%%%%%%%%%%%%%%%%%%%%%%%%%%%%%%%%%%%%%%%%%%%%%%%%%%%%%%
\begin{lemma}\label{lec1}
There is $C_1>0$ such that
\begin{equation}
\|u \|_{H^2} + \|u \|_{C^{3/2}(\cunb)}+ \lambda \int_{\unb} g(u(x))\ \mathrm{d}x  \le C_1\  \label{c1}
\end{equation} 
whenever $\lambda\in [0,\lambda_*]$ and $u\in\mathcal{S}_r^\lambda$.
\end{lemma}
%%%%%%%%%%%%%%%%%%%%%%%%%%%%%%%%%%%%%%%%%%%%%%%%%%%%%%%%%%%%%%%%%%%%

\begin{proof}
According to Lemma~\ref{L0} there are $m_1>0$ and $\phi_1\in C^4_r(\cunb)$ satisfying $A\phi_1=m_1\phi_1$ and
\begin{equation}
\phi_1 > 0  \;\;\text{ in }\;\; \unb\ , \quad \|\phi_1\|_1=1\ . \label{c3}
\end{equation}
Multiplying \eqref{stat1} by $\phi_1$ and integrating over $\unb$ give
$$
-m_1 \int_{\unb} \phi_1\ u\, \mathrm{d}x = \lambda \int_{\unb} g(u) \phi_1\, \mathrm{d}x\ .
$$
Since $u\ge -1$ in $\unb$, we deduce from \eqref{c3} that
\begin{equation}
0 \le \lambda \int_{\unb} g(u) \phi_1\, \mathrm{d}x \le m_1 \ . \label{c4}
\end{equation}
Next, recall that Lemma~\ref{leb1} ensures that the function $\mathfrak{u}~: [0,1] \to \mathbb{R}$ defined by $\mathfrak{u}(|x|)=u(x)$ for $x\in\cunb$ is non-decreasing. This readily implies that $g(\mathfrak{u})$ is non-increasing and, thanks to~\eqref{c4}, 
\begin{align*}
0 \le \lambda \int_{\unb} g(u(x))\ \mathrm{d}x & = \lambda |\uns| \int_0^1 g(\mathfrak{u}(r)) r^{d-1}\ \mathrm{d}r \\
& \le \lambda |\uns| \int_0^{3/4} g(\mathfrak{u}(r)) r^{d-1}\ \mathrm{d}r + \lambda |\uns| \int_{3/4}^1 g(\mathfrak{u}(r-1/4)) r^{d-1}\ \mathrm{d}r \\
& \le \lambda |\uns| \int_0^{3/4} g(\mathfrak{u}(r)) r^{d-1}\ \mathrm{d}r + 2^{d-1} \lambda  |\uns| \int_{1/2}^{3/4} g(\mathfrak{u}(r)) r^{d-1}\ \mathrm{d}r \\
& \le \lambda (1+2^{d-1}) |\uns| \int_0^{3/4} g(\mathfrak{u}(r)) r^{d-1}\ \mathrm{d}r \\
& \le  \frac{2^d \lambda}{\min_{\bar{\mathbb{B}}_{3/4}} \phi_1} \int_{\mathbb{B}_{3/4}} \phi_1(x) g(u(x))\ \mathrm{d}x \\
& \le \frac{2^d m_1}{\min_{\bar{\mathbb{B}}_{3/4}} \phi_1}\ .
\end{align*}
We have thus proved that
\begin{equation}
\lambda \| g(u)\|_1 \le C_1\ . \label{c5}
\end{equation}
It next follows from \eqref{stat1},\eqref{c5}, and the non-negativity of $g$ and $1+u$ that
\begin{equation}\label{obelix}
c\|u \|_{H^2} \le \langle u , u \rangle = - \lambda \int_{\unb} g(u) u\ \mathrm{d}x \le \lambda \|g(u)\|_1 \le C_1\ .
\end{equation}
Finally, if $d=1$, the embedding of $H^2(\unb)$ in $C^{3/2}(\cunb)$ completes the proof in this case.  If $d=2$, we note that $u$ solves $B\Delta^2 u=T\Delta u-\lambda g(u)$ in $\unb$ subject to homogeneous Dirichlet boundary conditions with $\|T\Delta u-\lambda g(u)\|_1\le C$ by \eqref{c5} and \eqref{obelix}. Hence, in this case the assertion follows from a version of the Brezis-Merle inequality \cite{BM91} (see Lemma~\ref{le.app1}) and the embedding of $W_q^2(\unb)$ in $C^{3/2}(\cunb)$ for $q$ large enough.
\end{proof}

Restricting our attention to $u\in\mathcal{S}_r^\lambda$ with a non-negative $\mu_1(u)$, the previous estimates can be improved in the following way.

%%%%%%%%%%%%%%%%%%%%%%%%%%%%%%%%%%%%%%%%%%%%%%%%%%%%%%%%%%%%%%%%%%%%
\begin{lemma}\label{leb4b}
There is $C_2>0$ such that
\begin{equation}
\| u \|_{H^2} + \int_{\unb} \frac{\mathrm{d}x}{(1+u(x))^3} \le C_2 \label{b39}
\end{equation} 
whenever $\lambda\in [0,\lambda_*]$ and $u\in\mathcal{S}_r^\lambda$ with $\mu_1(u)\ge 0$.
\end{lemma}
%%%%%%%%%%%%%%%%%%%%%%%%%%%%%%%%%%%%%%%%%%%%%%%%%%%%%%%%%%%%%%%%%%%%

\begin{proof}
We infer from \eqref{stat1} and the assumption $\mu_1(u)\ge 0$ that 
$$
-\lambda \int_{\unb} g(u) u\ \mathrm{d}x = \langle u , u \rangle \ge - \lambda \int_{\unb} g'(u) u^2\ \mathrm{d}x
$$
and thus
\begin{equation}
\int_{\unb} \frac{3 u^2 + u}{(1+u)^3}\ \mathrm{d}x \le 0\ .\label{b40}
\end{equation}
Observing that $3z^2+ z \ge 1/4$ for $z\in (-1,-1/2)$, we deduce from \eqref{b40} that
\begin{align*}
\int_{\unb} \frac{1}{(1+u)^3}\ \mathrm{d}x & \le \int_{\unb} \frac{\mathbf{1}_{(-1/2,\infty)}(u)}{(1+u)^3}\ \mathrm{d}x + \int_{\unb} \frac{\mathbf{1}_{(-1,-1/2)}(u)}{(1+u)^3}\ \mathrm{d}x \\
& \le 8 |\unb| + 4 \int_{\unb} \frac{\mathbf{1}_{(-1,-1/2)}(u) (3u^2+u)}{(1+u)^3}\,\rd x \le 8 |\unb|\ .
\end{align*}
Finally, \eqref{stat1} and H\"older's inequality give
$$
0 \le \langle u , u \rangle = - \lambda \int_{\unb} g(u) u\ \mathrm{d}x \le \lambda_* |\unb|^{1/3} \left( \int_{\unb} \frac{1}{(1+u)^3}\ \mathrm{d}x \right)^{2/3}\ ,
$$
and \eqref{b39} follows from the previous two inequalities and the finiteness of $\lambda_*$.
\end{proof}

%%%%%%%%%%%%%%%%%%%%%%%%%%%%%%%%%%%%%%%%%%%%%%%%%%%%%%%%%%%%%%%%%%%%
%%%%%%%%%%%%%%%%%%%%%%%%%%%%%%%%%%%%%%%%%%%%%%%%%%%%%%%%%%%%%%%%%%%%
\subsection{Maximal stationary solutions}\label{sec23}
%%%%%%%%%%%%%%%%%%%%%%%%%%%%%%%%%%%%%%%%%%%%%%%%%%%%%%%%%%%%%%%%%%%%
%%%%%%%%%%%%%%%%%%%%%%%%%%%%%%%%%%%%%%%%%%%%%%%%%%%%%%%%%%%%%%%%%%%%

We first recall the existence of maximal solutions to \eqref{stat1}-\eqref{stat2} established in \cite[Theorem~1.5]{LWxx}.

%%%%%%%%%%%%%%%%%%%%%%%%%%%%%%%%%%%%%%%%%%%%%%%%%%%%%%%%%%%%%%%%%%%%
\begin{proposition}\label{prb2}
(i) For any $\lambda\in (0,\lambda_*)$, the set $\mathcal{S}_r^\lambda$ is non-empty and contains a unique maximal element $u_\lambda$ in the sense that $u\le u_\lambda$ for all radially symmetric classical subsolutions $u$ to \eqref{stat1}-\eqref{stat2} with parameter~$\lambda$. In addition, for each $x\in \mathbb{B}_1$, the function $\lambda\mapsto u_\lambda (x)$ is decreasing in $(0,\lambda_*)$. 

\medskip

(ii) There is no radially symmetric classical solution to \eqref{stat1}-\eqref{stat2} for $\lambda>\lambda_*$.
\end{proposition}
%%%%%%%%%%%%%%%%%%%%%%%%%%%%%%%%%%%%%%%%%%%%%%%%%%%%%%%%%%%%%%%%%%%%

We supplement Proposition~\ref{prb2} with continuity properties of $\lambda\longmapsto u_\lambda$.

%%%%%%%%%%%%%%%%%%%%%%%%%%%%%%%%%%%%%%%%%%%%%%%%%%%%%%%%%%%%%%%%%%%%
\begin{lemma}\label{leb2b}
The map $\lambda\longmapsto u_\lambda$ is continuous from $[0,\lambda_*)$  to $C_r^2(\cunb)$ with $u_0=0$. In addition, the map $\lambda \longmapsto \mu_1(u_\lambda)$ belongs to $C([0,\lambda_*))$.
\end{lemma}
%%%%%%%%%%%%%%%%%%%%%%%%%%%%%%%%%%%%%%%%%%%%%%%%%%%%%%%%%%%%%%%%%%%%

\begin{proof}
Fix $\lambda\in [0,\lambda_*)$ and let $(\lambda_k)_{k\ge 1}$ be a sequence in $[0,\lambda_*)$ such that $\lambda_k\to \lambda$ as $k\to\infty$. Then there is $\eta\in (0,1)$ such that $$
\lambda<\eta\lambda_* \;\;\text{ and }\;\; \lambda_k\le \eta\lambda_*<\lambda_*\ , \quad k\ge 1\ .
$$ 
Proposition~\ref{prb2} ensures that $u_{\lambda_k} \ge u_{\eta\lambda_*}$ for all $k\ge 1$, so that $(u_{\lambda_k})_{k\ge 1}$ ranges in a compact subset of $(-1,0]$. Therefore, $(g(u_{\lambda_k}))_{k\ge 1}$ is bounded in $L_\infty(\unb)$ and classical regularity results entail that $(u_{\lambda_k})_{k\ge 1}$ is bounded in $W_q^4(\unb)$ for all $q\in (1,\infty)$, see \cite[Theorem~2.20]{GGS10} for instance. The compactness of Sobolev's embedding then implies that a subsequence of $(u_{\lambda_k})_{k\ge 1}$ (not relabeled) converges weakly in $H^4(\unb)$ and strongly in $C^3(\cunb)$ to a function $u\in H^4_{D,r}(\unb)$, which is a strong solution to \eqref{stat1}-\eqref{stat2} and satisfies $u\ge u_{\eta\lambda_*} > -1$ in $\cunb$. Since $g$ is smooth in $(-1,\infty)$, there is $\alpha>0$ such that $g(u)$ belongs to $C^{1+\alpha }(\cunb)$ and a further use of classical elliptic regularity results guarantees that $u$ actually belongs to $\mathcal{S}_r^\lambda$, see \cite[Theorem~2.19]{GGS10} for instance.

Consider now a radially symmetric classical subsolution $\sigma\in C_r^4(\unb)\cap C^2(\cunb)$ to \eqref{stat1}-\eqref{stat2} with $\sigma>-1$ in $\cunb$. For $\vartheta\in (0,1)$ and $k\ge 1$, we infer from the convexity of $g$ and the properties $\lambda_k\le \eta\lambda_*$ and  $g(u_{\eta\lambda_*})\ge 1$ in $\unb$ that 
\begin{align*}
 B \Delta^2 &\left( (1-\vartheta) \sigma + \vartheta u_{\eta\lambda_*} \right) - T \Delta \left( (1-\vartheta) \sigma + \vartheta u_{\eta\lambda_*} \right) + \lambda_k g\left( (1-\vartheta) \sigma + \vartheta u_{\eta\lambda_*} \right) \\
\le & (1-\vartheta) \left( B \Delta^2 \sigma - T \Delta \sigma + \lambda_k g(\sigma) \right) + \vartheta \left( B \Delta^2 u_{\eta\lambda_*} - T \Delta u_{\eta\lambda_*} + \lambda_k g(u_{\eta\lambda_*}) \right) \\
\le & (1-\vartheta) (\lambda_k-\lambda) g(\sigma) + \vartheta (\lambda_k - \eta\lambda_*) g(u_{\eta\lambda_*}) \\
\le & |\lambda_k - \lambda| \|g(\sigma)\|_\infty - \vartheta (\eta\lambda_* - \lambda_k)\ .
\end{align*}
Since $\lambda_k\to \lambda$ as $k\to\infty$ and $\lambda<\eta\lambda_*$, there is $k_\vartheta\ge 1$ large enough such that, for all $k\ge k_\vartheta$, 
$$
|\lambda_k - \lambda| \|g(\sigma)\|_\infty - \vartheta (\eta\lambda_* - \lambda_k) \le 0\ ,
$$
and hence $(1-\vartheta) \sigma + \vartheta u_{\eta\lambda_*}$ is a subsolution to \eqref{stat1}-\eqref{stat2} with parameter $\lambda_k$. Therefore, owing to the maximality property of $u_{\lambda_k}$,
$$
(1-\vartheta) \sigma + \vartheta u_{\eta\lambda_*} \le u_{\lambda_k} \;\;\text{ in }\;\; \unb 
$$ 
for all $k\ge k_\vartheta$. We first let $k\to\infty$ and then $\vartheta\to 0$ in the above inequality to conclude that $\sigma\le u$ in $\unb$. In other words, $u$ is a maximal solution to \eqref{stat1}-\eqref{stat2} and thus $u=u_\lambda$. 

Owing to the definition \eqref{b6}, the continuity of $\lambda\longmapsto \mu_1(u_\lambda)$ in $[0,\lambda_*)$ readily follows from that of $\lambda\longmapsto u_\lambda$ which we have just established.
\end{proof}

The next proposition entails that the maximal solutions are exactly the linearly stable solutions.
 
%%%%%%%%%%%%%%%%%%%%%%%%%%%%%%%%%%%%%%%%%%%%%%%%%%%%%%%%%%%%%%%%%%%%
\begin{proposition}\label{thb5}
Let $\lambda\in [0,\lambda_*)$. Then $\mu_1(u_\lambda)>0$, and if $u\in\mathcal{S}_r^\lambda$ satisfies $\mu_1(u)\ge 0$, then $u=u_\lambda$.
\end{proposition}
%%%%%%%%%%%%%%%%%%%%%%%%%%%%%%%%%%%%%%%%%%%%%%%%%%%%%%%%%%%%%%%%%%%%

\begin{proof}
Due to the monotonicity and negativity of $g'$ and the monotonicity of $\lambda\longmapsto u_\lambda$ stated in Proposition~\ref{prb2}, it readily follows from \eqref{b6} that 
\begin{equation*}
\mu_1(u_{\lambda_1}) \ge \mu_1(u_{\lambda_2}) \;\;\text{ for }\;\; 0 \le \lambda_1 \le \lambda_2 < \lambda_*\ . \label{c80}
\end{equation*}
Introducing 
$$
\lambda_{st} := \sup{\left\{ \lambda \in [0,\lambda_*)\, : \, \mu_1(u_\lambda) > 0 \right\}}\ ,
$$
we assume for contradiction that $\lambda_{st}<\lambda_*$. Lemma~\ref{leb2b} then ensures that $\mu_1(u_{\lambda_{st}})=0$. Now, given $\lambda\in (\lambda_{st},\lambda_*)$, we deduce from \eqref{stat1} that 
$$
B \Delta^2 u_\lambda - T \Delta u_\lambda + \lambda_{st} g(u_\lambda) = (\lambda_{st} - \lambda) g(u_\lambda) < 0\  \text{ in }\ \unb\ .
$$Applying Lemma~\ref{leb4}~(ii), we conclude that $u_\lambda=u_{\lambda_{st}}$ and a contradiction. Therefore, \mbox{$\lambda_{st}=\lambda_*$}.
Finally, considering $u\in\mathcal{S}_r^\lambda$ such that $\mu_1(u)\ge 0$, Lemma~\ref{leb4}~(i)  implies  $u_\lambda\le u$ while the maximal property of $u_\lambda$ guarantees  $u\le u_\lambda$. Therefore, $u=u_\lambda$.
\end{proof}

\medskip

We now show that the maximal arc-connected set $\mathcal{A}_0$ defined in \eqref{A0} coincides with the branch of maximal solutions $(\lambda,u_\lambda)$, $\lambda\in (0,\lambda_*)$.

%%%%%%%%%%%%%%%%%%%%%%%%%%%%%%%%%%%%%%%%%%%%%%%%%%%%%%%%%%%%%%%%%%%%
\begin{corollary}\label{C20}
There holds $\lambda_0=\lambda_*$ and $V(\lambda)=u_\lambda$ for each $\lambda\in (0,\lambda_*)$. Moreover, $\lim_{\lambda\rightarrow\lambda_*} \mu_1(u_\lambda)=0$.
\end{corollary}
%%%%%%%%%%%%%%%%%%%%%%%%%%%%%%%%%%%%%%%%%%%%%%%%%%%%%%%%%%%%%%%%%%%%

\begin{proof}
Since $V(\lambda)$ is for each $\lambda\in (0,\lambda_0)$ a classical radially symmetric solution to \eqref{stat1}-\eqref{stat2}, we clearly have $\lambda_0\le\lambda_*$. We now claim that
\begin{equation}\label{a1u}
\mu_1(V(\lambda))>0\ \text{ for } \lambda\in [0,\lambda_0)\ .
\end{equation}
Indeed, the continuity of $V$ entails that $\lambda\mapsto \mu_1(V(\lambda))$ is continuous on $[0,\lambda_0)$ with $$\mu_1(V(0))=\mu_1(0)=m_1>0$$ with $m_1$ introduced in Lemma~\ref{L0}. Clearly, Lemma~\ref{leb3c} (ii) implies $\mu_1(V(\lambda))>0$ for each $\lambda\in [0,\lambda_0)$. Consequently, $u_\lambda=V(\lambda)$ for each $\lambda\in [0,\lambda_0)$ by Proposition~\ref{thb5}. 

Now, suppose for contradiction that $\lambda_0<\lambda_*$. Then $V(\lambda)\rightarrow u_{\lambda_0}$ in $C_r^2(\cunb)$ as $\lambda\rightarrow\lambda_0$ according to Lemma~\ref{leb2b}. Owing to the maximality of $\lambda_0$, this implies that $F_u(\lambda_0,u_{\lambda_0})$ is not boundedly invertible which contradicts Proposition~\ref{thb5} since $\mu_1(u_{\lambda_0})>0$. In particular, $\lim_{\lambda\rightarrow\lambda_*} \mu_1(u_\lambda)=0$ due to the maximality $\lambda_*$.
\end{proof}

We finally investigate the behavior of $u_\lambda$ as $\lambda\to\lambda_*$. 

%%%%%%%%%%%%%%%%%%%%%%%%%%%%%%%%%%%%%%%%%%%%%%%%%%%%%%%%%%%%%%%%%%%%
\begin{proposition}\label{prb7}
For $x\in\cunb$, define
\begin{equation}
u_{\lambda_*}(x) := \inf_{\lambda\in (0,\lambda_*)} u_\lambda(x) = \lim_{\lambda\to\lambda_*} u_\lambda(x) \in [-1,0]\ . \label{b500}
\end{equation}
Then $u_{\lambda_*}\in\mathcal{S}_r^{\lambda_*}$ and $\mu_1(u_{\lambda_*}) = 0$.  Moreover, any $u\in \mathcal{S}_r^{\lambda_*}$ satisfies $u\le u_{\lambda_*}$ in $\unb$ and if, in addition, $\mu_1(u)\ge 0$, then $u=u_{\lambda_*}$.
\end{proposition}
%%%%%%%%%%%%%%%%%%%%%%%%%%%%%%%%%%%%%%%%%%%%%%%%%%%%%%%%%%%%%%%%%%%%

\begin{proof}
The proof is similar to that of \cite[Theorem~11.4.1]{EGG10}. 
Indeed, the fact that $u_{\lambda_*}$ is well-defined is a simple consequence of Proposition~\ref{prb2} (i). Thanks to Proposition~\ref{thb5}, we are in a position to apply Lemma~\ref{leb4b} and conclude that $(u_\lambda)_\lambda$ is bounded in $H^2(\unb)$ while $(g(u_\lambda))_\lambda$ is bounded in $L_{3/2}(\unb)$. Consequently, the sequence $(u_\lambda)_\lambda$ is bounded in $W^4_{3/2}(\unb)$ by classical elliptic regularity, see \cite[Theorem~2.20]{GGS10} for instance, so that $u_{\lambda_*}\in C^1(\cunb)$ due to the continuous embedding of $W^4_{3/2}(\unb)$ in $C^1(\cunb)$. If the minimum of $u_{\lambda_*}$ in $\unb$ would be equal to $-1$, then $u_{\lambda_*}(0)=-1$ and $\nabla u_{\lambda_*}(0)=0$ according to Lemma~\ref{leb1}. These properties entail that there is $C>0$ such that
$$
\frac{1}{1+u_{\lambda_*}(x)} \ge \frac{C}{|x|}\ ,\quad x\in\cunb\ ,
$$
which contradicts the boundedness of $(g(u_\lambda))_\lambda$ in $L_{3/2}(\unb)$. Therefore, $u_{\lambda_*}>-1$ in $\cunb$ which, together with the above mentioned estimates and classical elliptic regularity, implies $u_{\lambda_*}\in \mathcal{S}_r^{\lambda_*}$. 
That $\mu_1(u_{\lambda_*})= 0$ follows from Corollary~\ref{C20}.

Finally, take $u\in \mathcal{S}_r^{\lambda_*}$. Then, for all  $\lambda\in (0,\lambda_*)$, 
$$
B \Delta^2 u - T \Delta u + \lambda g(u) = (\lambda-\lambda_*) g(u) \le 0 \;\;\text{ in }\;\; \unb\ ,
$$
and thus $u\le u_\lambda$ in $\unb$ by Proposition~\ref{prb2}. Letting $\lambda\to\lambda_*$ gives $u\le u_{\lambda_*}$. The uniqueness assertion is a consequence of Lemma~\ref{leb4} and the maximality of $u_{\lambda_*}$.
\end{proof}

%%%%%%%%%%%%%%%%%%%%%%%%%%%%%%%%%%%%%%%%%%%%%%%%%%%%%%%%%%%%%%%%%%%%
%%%%%%%%%%%%%%%%%%%%%%%%%%%%%%%%%%%%%%%%%%%%%%%%%%%%%%%%%%%%%%%%%%%%
\subsection{Continuation of maximal solutions}\label{sec24}
%%%%%%%%%%%%%%%%%%%%%%%%%%%%%%%%%%%%%%%%%%%%%%%%%%%%%%%%%%%%%%%%%%%%
%%%%%%%%%%%%%%%%%%%%%%%%%%%%%%%%%%%%%%%%%%%%%%%%%%%%%%%%%%%%%%%%%%%%

We shall now completely characterize the radially symmetric solutions to \eqref{stat1}-\eqref{stat2} near $(\lambda_*,u_{\lambda_*})$. 
That is, we show that the curve $\mathcal{A}$ defined in~\eqref{A} non-trivially extends 
the curve $\{(\lambda,u_\lambda)\,:\,\lambda\in(0,\lambda_*)\}$ of maximal solutions established in Subsection~\ref{sec23}, the latter coinciding with the curve $\mathcal{A}_0$ from \eqref{A0} as proven in Corollary~\ref{C20}. Moreover, all radially symmetric solutions to \eqref{stat1}-\eqref{stat2} near $(\lambda_*,u_{\lambda_*})$ lie on this curve, which in particular excludes any bifurcation phenomenon near this point.

Indeed, let us first note  that $(\lambda_*,u_{\lambda_*})$ cannot be the limit of $(\Lambda(s),U(s))$ as $s\rightarrow\infty$ owing to Proposition~\ref{prb7} and Theorem~\ref{T19} (iii). Thus, there is a minimal $s_*\in [1,\infty)$ such that 
\begin{equation}\label{AA}
(\Lambda(s_*),U(s_*))=(\lambda_*,u_{\lambda_*})\ \text{ and }\ \mathcal{A}_0= (\Lambda,U)((0,s_*))\ .
\end{equation}
Next recall from Proposition~\ref{prb7} that $u_{\lambda_*}\in\mathcal{S}_r^{\lambda_*}$ with $\mu_1(u_{\lambda_*})=0$ and it thus readily follows from Lemma~\ref{leb3c} that the kernel of the partial Fr\'echet derivative $F_u(\lambda_*,u_{\lambda_*})$ of the function $F$ defined in \eqref{F} is spanned by a positive function $\phi_*\in C_{D,r}^4(\cunb)$.

Now, a precise description of the behavior of $\mathcal{A}$ near $(\lambda_*,u_{\lambda_*})$ can be obtained from \cite[Theorem 3.2]{CR73} and is stated in the following theorem.

%%%%%%%%%%%%%%%%%%%%%%%%%%%%%%%%%%%%%%%%%%%%%%%%%%%%%%%%%%%%
\begin{theorem}\label{T21}
There are $\delta>0$, $\varepsilon>0$, and an injective and continuous function $\zeta$ from $(-\delta,\delta)$ onto $(s_*-\varepsilon,s_*+\varepsilon)$ with the following properties:
\begin{itemize}
\item[(i)] $\zeta(0)=s_*$;
\item[(ii)] $(\Lambda,U)\circ \zeta$ is a real analytic function on $(-\delta,\delta)$;
\item[(iii)] all solutions $(\lambda,u)$ to \eqref{stat1}-\eqref{stat2} near the point $(\lambda_*,u_{\lambda_*})=(\Lambda(s_*),U(s_*))$ lie on the curve $\{(\Lambda,U)\circ\zeta(\sigma)\,:\, \vert \sigma\vert <\delta\}$;
\item[(iv)]  $(\Lambda\circ\zeta)'(0)=0$ and~$(U\circ\zeta)'(0)=\phi_*$.
\end{itemize}
\end{theorem}
%%%%%%%%%%%%%%%%%%%%%%%%%%%%%%%%%%%%%%%%%%%%%%%%%%%%%%%%%%%%

\begin{proof}
 Since $u_{\lambda_*}\in\mathcal{S}_r^{\lambda_*}$ we have $u_{\lambda_*}\in \mathcal{O}$ with $\mathcal{O}$ defined in \eqref{O} and thus the function $F$ is analytic near $(\lambda_*,u_{\lambda_*})$. As the kernel of $F_u(\lambda_*,u_{\lambda_*})$ is one-dimensional, $\mathrm{codim}(\mathrm{rg}(F_u(\lambda_*,u_{\lambda_*})))$ equals 1 since $F_u(\lambda_*,u_{\lambda_*})$ is a Fredholm operator of index 0, see e.g. \cite[Theorem 4.25]{RudinFA}. We now claim that
$$
F_\lambda(\lambda_*,u_{\lambda_*})=A^{-1} g(u_{\lambda_*}) \not\in \mathrm{rg}(F_u(\lambda_*,u_{\lambda_*}))\ .
$$
Indeed, if otherwise there is  $\varphi\in  C_{D,r}^{1 }(\cunb)$ with
$$
A^{-1} g(u_{\lambda_*})=\varphi+ \lambda_* A^{-1} g'(u_{\lambda_*})\varphi\ .
$$
But then $\varphi\in H_{D,r}^4(\unb)$ satisfies
$$
B\Delta^2\varphi-T\Delta\varphi+ \lambda_*  g'(u_{\lambda_*})\varphi= g(u_{\lambda_*})\ \text{ in } \ \unb\ ,
$$
and testing this equation with $\phi_*>0$ yields the contradiction
$$
0=\int_{\unb}\big[B\Delta^2\phi_*-T\Delta \phi_*+\lambda_*  g'(u_{\lambda_*})\phi_* \big]\, \varphi\,\rd x=\int_{\unb} g(u_{\lambda_*})\,\phi_*\,\rd x > 0\ .
$$
Therefore, we are in a position to apply \cite[Theorem 3.2]{CR73} and obtain  in combination with Theorem~\ref{T19} the assertion.
\end{proof}

Actually, the curve $\mathcal{A}$ bends down at  $(\lambda_*,u_{\lambda_*})$:

%%%%%%%%%%%%%%%%%%%%%%%%%%%%%%%%%%%%%%%%%%%%%%%%%%%%%%%%%%%%%%%%%%%%
\begin{corollary}\label{C22}
There holds $(\Lambda\circ\zeta)''(0)<0$.
\end{corollary}
%%%%%%%%%%%%%%%%%%%%%%%%%%%%%%%%%%%%%%%%%%%%%%%%%%%%%%%%%%%%%%%%%%%%

\begin{proof}
Twice differentiation of the equality
$$
0=F(\Lambda\circ\zeta(\sigma),U\circ\zeta(\sigma))=U\circ\zeta(\sigma)+\Lambda\circ\zeta(\sigma) A^{-1} g(U\circ\zeta(\sigma))
$$
with respect to $\sigma$ at $\sigma=0$ and Theorem~\ref{T21} yield
$$
(\Lambda\circ\zeta)''(0) A^{-1} g(u_{\lambda_*}) +\lambda_* A^{-1} g''(u_{\lambda_*})\phi_*^2=- F_u(\lambda_*, u_{\lambda_*}) (U\circ\zeta)''(0)\ ,
$$
that is,
$$
(\Lambda\circ\zeta)''(0)  g(u_{\lambda_*}) +\lambda_* g''(u_{\lambda_*})\phi_*^2= - A (U\circ\zeta)''(0)-\lambda_* g'(u_{\lambda_*}) (U\circ\zeta)''(0)\ .
$$
Testing this last equation with $\phi_*>0$ and using the convexity of $g$ imply
$$
(\Lambda\circ\zeta)''(0) \int_{\unb} g(u_{\lambda_*})\,\phi_*\,\rd x=-\lambda_*  \int_{\unb}  g''(u_{\lambda_*})\phi_*^3\,\rd x <0
\ ,$$
whence $(\Lambda\circ\zeta)''(0)<0$.
\end{proof}

%%%%%%%%%%%%%%%%%%%%%%%%%%%%%%%%%%%%%%%%%%%%%%%%%%%%%%%%%%%%%%%%%%%%
%%%%%%%%%%%%%%%%%%%%%%%%%%%%%%%%%%%%%%%%%%%%%%%%%%%%%%%%%%%%%%%%%%%%
\subsection{End point}\label{sec25}
%%%%%%%%%%%%%%%%%%%%%%%%%%%%%%%%%%%%%%%%%%%%%%%%%%%%%%%%%%%%%%%%%%%%
%%%%%%%%%%%%%%%%%%%%%%%%%%%%%%%%%%%%%%%%%%%%%%%%%%%%%%%%%%%%%%%%%%%%

The following theorem now completes the picture of the curve $\mathcal{A}$ defined in~\eqref{A}. It characterizes the limit of $(\Lambda(s),U(s))$ as $s\rightarrow\infty$ and shows that for each $\lambda\in (0,\lambda_*)$ there are at least two steady-states.

%%%%%%%%%%%%%%%%%%%%%%%%%%%%%%%%%%%%%%%%%%%%%%%%%%%%%%%%%%%%%%%%%%%%
\begin{theorem}\label{C217}
(i) As $s\rightarrow\infty$, $$(\Lambda(s),U(s))\rightarrow (0,\omega)\ \text{ in }\ \RR\times \big(C(\cunb)\cap C^1(\cunb\setminus\mathbb{B}_\rho)\big)$$ for each $\rho\in (0,1)$, where $\omega\in C^4(\cunb\setminus\{0\})\cap C_r^1(\cunb)$ solves the equation
$$
B \Delta^2 \omega(x) - T \Delta \omega(x)  = 0 \;\;\text{ for }\;\; x\in \mathbb{B}_1\setminus\{0\}\ , 
$$
with boundary conditions
\begin{align*}
\omega(x) = \partial_\nu \omega(x) & = 0 \;\;\text{ for }\;\; x\in  \uns\ , \\
\omega(0)+1 = \nabla \omega(0) & = 0 \ , 
\end{align*}
and satisfies $\omega(x)>-1$ for $0<\vert x\vert <1$.

\medskip

(ii) For each $\lambda\in (0,\lambda_*)$ there are at least two values $0<s_1<s_*<s_2$ with $\Lambda(s_j)=\lambda$, $U(s_j)\in\mathcal{S}_r^\lambda$ for $j=1,2$, and $U(s_2)\le U(s_1)$ in $\unb$ with $U(s_2)\not= U(s_1)$.
\end{theorem}
%%%%%%%%%%%%%%%%%%%%%%%%%%%%%%%%%%%%%%%%%%%%%%%%%%%%%%%%%%%%%%%%%%%%

 Note that Theorem~\ref{C217} allows one to compute the end point $\omega$ explicitly in terms of the modified Bessel functions of the first and second kinds for $d=2$ and, respectively, in terms of the exponential function for $d=1$ (cf. Figure~\ref{fig1}).

\medskip

To prove Theorem~\ref{C217}, we first need the following result relating the minimum of a function $w$ to the integrability of $g(w)$. 

%%%%%%%%%%%%%%%%%%%%%%%%%%%%%%%%%%%%%%%%%%%%%%%%%%%%%%%%%%%%%%%%%%%%
\begin{lemma}\label{lec2}
Let $w$ be a radially symmetric function in $H^2_{D,r}(\unb)$ such that $g(w)\in L_1(\unb)$, and let the profile $\mathfrak{w}$ of $w$ defined by $\mathfrak{w}(|x|) := w(x)$ for $x\in\cunb$ be a non-decreasing function on $[0,1]$. Then there is $C_3>0$  such that
\begin{equation}
I_{d}(1+w(0)) \le C_3 \left( 1 + \| w \|_{H^2}^2 \right) \| g(w)\|_1\ , \label{c6}
\end{equation} 
where $d=1,2$ and
$$
I_{1}(z)  : =   \frac{1}{z^{4/3}} - 1 \ , \qquad I_{2}(z) : = -\ln{z}\ , \qquad z\in (0,1)\ .
$$
\end{lemma}
%%%%%%%%%%%%%%%%%%%%%%%%%%%%%%%%%%%%%%%%%%%%%%%%%%%%%%%%%%%%%%%%%%%%

\begin{proof}
Since $\mathfrak{w}$ is non-decreasing on $[0,1]$, we deduce from the Cauchy-Schwarz inequality that, for $r\in [0,1]$,
\begin{align*}
0 \le r^{d-1} \partial_r \mathfrak{w}(r) & = \int_0^r \partial_s \left( s^{d-1} \partial_s \mathfrak{w}(s) \right)\ \mathrm{d}s \\
& \le \sqrt{\frac{r^d}{d}} \left( \int_0^r \left| \frac{1}{s^{d-1}} \partial_s \left( s^{d-1} \partial_s \mathfrak{w}(s) \right) \right|^2 s^{d-1}\ \mathrm{d}s \right)^{1/2} \\
& \le \frac{r^{d/2}}{\sqrt{d |\uns|}} \|\Delta w\|_2\ .
\end{align*}
Then $0 \le \partial_r \mathfrak{w}(r) \le C \|\Delta w\|_2 r^{(2-d)/2}$ and integrating once more with respect to $r$ gives 
$$\mathfrak{w}(r) \le \mathfrak{w}(0) + C \|\Delta w\|_2\ r^{(4-d)/2}\ ,\quad r\in [0,1]\ .
$$ 
Consequently,
$$
\int_{\unb} g(w(x)) \mathrm{d}x \ge \int_{\unb} g\left( w(0) + C \|\Delta w\|_2 \ |x|^{(4-d)/2} \right)\ \mathrm{d}x\ .
$$
Setting $\varrho := (1+w(0))^{2/(4-d)}$ and restricting the integral on the right-hand side of the above inequality to $\unb\setminus\mathbb{B}_\varrho$, we obtain
\begin{align*}
\int_{\unb} g(w(x)) \mathrm{d}x & \ge \int_{\unb\setminus\mathbb{B}_\varrho} \frac{1}{\left( 1+ w(0) + C \|\Delta w\|_2 \ |x|^{(4-d)/2} \right)^{2}}\ \mathrm{d}x \\
& \ge \int_{\unb\setminus\mathbb{B}_\varrho} \frac{1}{\left( 1 + C \|\Delta w\|_2 \right)^{2}}\ \frac{\mathrm{d}x}{|x|^{(4-d)}} \\
& \ge \frac{C}{\left( 1+\|\Delta w\|_2^2 \right)}\ I_{d}(1+w(0))\ ,
\end{align*}
whence \eqref{c6}.
\end{proof}

%%%%%%%%%%%%%%%%%%%%%%%%%%%%%%%%%%%%%%%%%%%%%%%%%%%%%%%%%%%%%%%%%%%%
\begin{lemma}\label{lec3}
Let $(\lambda_n)_{n\ge 1}$ be a sequence of real numbers in $[0,\lambda_*]$ and $(v_n)_{n\ge 1}$ be such that $v_n\in\mathcal{S}_r^{\lambda_n}$ for each $n\ge 1$. If 
\begin{equation}
\lim_{n\to \infty} \min_{\cunb} v_n = -1\ , \label{c7}
\end{equation}
then there are a subsequence of $(\lambda_n, v_n)_{n\ge 1}$ (not relabeled) and $\omega\in C_r^4(\cunb\setminus\{0\}) \cap C_r(\cunb)$ such that $\omega$ solves
\begin{align}
B \Delta^2 \omega(x) - T \Delta \omega(x) & = 0 \;\;\text{ for }\;\; x\in \mathbb{B}_1\setminus\{0\}\ , \label{c12} \\
\omega(x) = \partial_\nu \omega(x) & = 0 \;\;\text{ for }\;\; x\in  \uns\ , \label{c12b}
\end{align}
and
\begin{align}
\lim_{n\to\infty} \lambda_n & = 0 \ , \label{c8} \\
\lim_{n\to\infty} \left\| v_n - \omega \right\|_{C^{1+\alpha}(\cunb)} & = = 0\ , \qquad \alpha\in [0,1/2)\ , \label{c9} \\ 
v_n \rightharpoonup \omega &\ \text{ in }\;\; H^2(\unb)\ , \label{c10} \\
\omega(0) = -1  &\ \text{ and }\;\; \omega(x)>-1 \;\;\text{ for }\;\; x\in \mathbb{B}_1\setminus\{0\}\ . \label{c11} 
\end{align}
\end{lemma}
%%%%%%%%%%%%%%%%%%%%%%%%%%%%%%%%%%%%%%%%%%%%%%%%%%%%%%%%%%%%%%%%%%%%

\begin{proof}
{\bf Step 1: Compactness.} By Lemma~\ref{lec1} and the finiteness of $\lambda_*$, $(v_n)_n$ is bounded in $H^2(\unb)\cap C^{3/2}(\cunb)$ and $(\lambda_n)_n$ is bounded in $[0,\lambda_*]$. The compactness of the embedding of $ C^{3/2}(\cunb)$ in $C^{1+\alpha}(\cunb)$ for $\alpha\in [0,1/2)$ guarantees that, after possibly extracting a subsequence, we may assume that there are $\lambda_\infty\in [0,\lambda_*]$ and $\omega\in H^2(\unb)\cap C^{1+\alpha}(\cunb)$ for $\alpha\in [0,1/2)$ such that $\lambda_n\to \lambda_\infty$ as $n\to\infty$ and \eqref{c9} and \eqref{c10} hold true. Combining \eqref{c7} and \eqref{c9} readily gives
\begin{equation}
\omega(0) = -1 \,  ,\quad \nabla\omega(0)=0\, ,\;\; \text{ and }\;\; \omega=\partial_\nu \omega = 0 \;\;\text{ on }\;\; \uns\ . \label{c14}
\end{equation} 
In addition, by Lemma~\ref{leb1}, $v_n$ is radially symmetric with a non-decreasing profile $\mathfrak{v}_n$ defined by $\mathfrak{v}_n(|x|) := v_n(x)$ for $x\in\cunb$. Consequently, the function $\omega$ enjoys the same properties by \eqref{c9} and its profile $\mathfrak{v}$, defined by $\mathfrak{v}(|x|) := \omega(x)$ for $x\in\cunb$, is a non-decreasing function on $[0,1]$. Therefore, it follows from this property and~\eqref{c14} that there is $a\in [0,1)$ such that  
\begin{equation}
\omega(x) = -1 \;\;\text{ for }\;\; x\in\bar{\mathbb{B}}_a\;\; \text{ and }\;\; \omega(x)>-1 \;\;\text{ for }\;\; x\in \mathbb{B}_1\setminus\bar{\mathbb{B}}_a\ . \label{c111} \\
\end{equation}
In addition, if $a>0$, then
\begin{equation}
\partial_\nu \omega(x) = 0 \;\;\text{ for }\;\; x\in  \partial\mathbb{B}_a\ . \label{c13}
\end{equation}

{\bf Step 2: Identification of $\lambda_\infty$.}  To this end, we apply Lemma~\ref{lec2} and use Lemma~\ref{lec1} to obtain
\begin{equation}
\lambda_n I_d(1+v_n(0)) \le C_3 \left( 1 + \|v_n\|_{H^2}^2 \right) \lambda_n \|g(v_n)\|_1 \le C_1 C_3 (1+C_1^2)\ , \label{c15}
\end{equation}
which also reads
$$
\lambda_n \left( \frac{1 - (1+v_n(0))^{4/3}}{(1+v_n(0))^{4/3}} \right) \le C \;\;\text{ if }\;\; d=1 \;\;\text{ and }\;\; \lambda_n |\ln{(1+v_n(0))}| \le C \;\;\text{ if }\;\; d=2\ .
$$
Letting $n\to\infty$ in the above inequality readily gives $\lambda_\infty=0$ by \eqref{c7}, whence \eqref{c8}.

Next, fix $\varrho\in (a,1)$. We infer from \eqref{c9} and \eqref{c111} that  $(g(v_n))_{n\ge 1}$ is bounded in $L_\infty(\unb\setminus\bar{\mathbb{B}}_\varrho)$, so that $(\lambda_n g(v_n))_{n\ge 1}$ converges to zero in $L_\infty(\unb\setminus\bar{\mathbb{B}}_\varrho)$ as $n\to\infty$ by \eqref{c8}. Classical elliptic regularity estimates then allow us to pass to the limit as $n\to\infty$ and conclude that $\omega\in C^4(\cunb\setminus\mathbb{B}_\varrho)$ satisfies $B \Delta^2 \omega(x) - T \Delta \omega(x) = 0$ for all $x\in \unb \setminus \bar{\mathbb{B}}_\varrho$ and \eqref{c12b}. Since $\varrho$ is arbitrary in $(a,1)$, we have shown 
\begin{equation}
B \Delta^2 \omega(x) - T \Delta \omega(x)  = 0 \;\;\text{ for }\;\; x\in \mathbb{B}_1\setminus\bar{\mathbb{B}}_a\ . \label{c122} 
\end{equation}

{\bf Step 3: Identification of $a$.} The final step is to prove that the yet unknown number $a$  is equal to zero. 

Let $n\ge 1$. According to \eqref{b5}, there is $r_n\in (0,1)$ such that
\begin{equation}
\Delta v_n < 0 \;\;\text{ in }\;\; \cunb\setminus\bar{\mathbb{B}}_{r_n} \;\;\text{ and }\;\; \Delta v_n > 0 \;\;\text{ in }\;\; \mathbb{B}_{r_n}\ . \label{c20}
\end{equation} 
The boundary conditions for $v_n$ and Lemma~\ref{lec1} then imply 
\begin{align}
\int_{\unb\setminus\bar{\mathbb{B}}_{r_n}} |\Delta v_n(x)|\ \mathrm{d}x & = - \int_{\unb\setminus\bar{\mathbb{B}}_{r_n}} \Delta v_n(x)\ \mathrm{d}x = -\int_{\unb} \Delta v_n(x)\ \mathrm{d}x + \int_{\mathbb{B}_{r_n}} \Delta v_n(x)\ \mathrm{d}x \nonumber \\
& \le \sqrt{|\mathbb{B}_{r_n}|}\, \|\Delta v_n\|_2 \le C_1 \sqrt{|\mathbb{B}_{r_n}|}\ . \label{c21}
\end{align}
Assume for contradiction that there is a subsequence $(r_{n_k})_k$ of $(r_n)_n$ such that $r_{n_k}\to 0$ as $k\to\infty$. It readily follows from \eqref{c21} that $(\Delta v_{n_k})_k$ converges to zero in $L_1(\unb\setminus \bar{\mathbb{B}}_\varrho)$ for all $\varrho\in (0,1)$. Recalling \eqref{c10}, we deduce that $\Delta \omega = 0$ almost everywhere in $\unb$ which, together with \eqref{c12b}, implies $\omega\equiv 0$ in $\unb$ and contradicts \eqref{c111}. Therefore, there is $r_\star>0$ such that
\begin{equation}
r_n \ge r_\star>0\ , \qquad n\ge 1\ . \label{c22}
\end{equation} 

Now, for $n\ge 1$ and $x\in\cunb$, we set $\mathfrak{v}_n(|x|)=v_n(x)$, $\mathfrak{w}_n(|x|) := \Delta v_n(x)$, and define
$$
\sigma_n := \sup{\left\{ r\in (0,r_\star)\ :\ \mathfrak{v}_n(r) < - 1 + \sqrt{\lambda_n} \right\}}
$$
if the set is non-empty, and $\sigma_n=0$ otherwise. Since $\sigma_n\in [0,r_\star]$, we may assume, after possibly extracting a subsequence, that
\begin{equation}
\lim_{n\to\infty} \sigma_n = \sigma \in [0,r_\star]\ . \label{c23}
\end{equation}
Assume  $\sigma>0$ for contradiction. The definition of $\sigma_n$, \eqref{c9}, \eqref{c20}, and \eqref{c22} then ensure that 
\begin{equation}
\omega(x)= -1 \;\;\text{ for }\;\; x\in\mathbb{B}_\sigma \;\;\text{ and }\;\; \Delta v_n(x) \ge 0 \;\;\text{ for }\;\; x\in\mathbb{B}_{\sigma_n}\ . \label{c24}
\end{equation}
Since $1+v_n(x)\le \sqrt{\lambda_n}$ for $x\in \mathbb{B}_{\sigma_n}$ and $v_n\in\mathcal{S}_r^{\lambda_n}$, 
we find 
$$
B \Delta^2 v_n - T \Delta v_n = -\lambda_n g(v_n) \le - 1 \;\;\text{ in }\;\; \mathbb{B}_{\sigma_n}\ .
$$
Consequently,
$$
\partial_r \left( r^{d-1} \left( B \partial_r \mathfrak{w}_n(r) - T \partial_r \mathfrak{v}_n(r) \right) \right) \le - r^{d-1}\ , \qquad r\in (0,\sigma_n)\ ,
$$
and, because $r^{d-1} \partial_r \mathfrak{w}_n(r)$ and $r^{d-1} \partial_r \mathfrak{v}_n(r)$ both vanish as $r\to 0$, a first integration gives
$$
\partial_r \left( B \mathfrak{w}_n(r) - T \mathfrak{v}_n(r) \right) \le - \frac{r}{d}\ , \qquad r\in (0,\sigma_n)\ .
$$
We next integrate the above differential inequality over $(r,\sigma_n)$ to obtain
$$
B \mathfrak{w}_n(r) \ge B \mathfrak{w}_n(\sigma_n) + T \left( \mathfrak{v}_n(r) - \mathfrak{v}_n(\sigma_n) \right) + \frac{\sigma_n^2}{2d} - \frac{r^2}{2d}\ .
$$
Due to $\mathfrak{w}_n(\sigma_n)\ge 0$ by \eqref{c24}, we find
$$
B \partial_r \left( r^{d-1} \partial_r \mathfrak{v}_n(r) \right) =B r^{d-1}\mathfrak{w}_n(r) \ge T r^{d-1} \left( \mathfrak{v}_n(r) - \mathfrak{v}_n(\sigma_n) \right) + \frac{\sigma_n^2 r^{d-1}}{2d} - \frac{r^{d+1}}{2d}\ ,
$$
whence, after integrating once more, 
$$
B r^{d-1} \partial_r \mathfrak{v}_n(r) \ge T \int_0^r s^{d-1} \left( \mathfrak{v}_n(s) - \mathfrak{v}_n(\sigma_n) \right) \mathrm{d}s + \frac{\sigma_n^2 r^{d}}{2d^2} - \frac{r^{d+2}}{2d(d+2)}\ , \qquad r\in (0,\sigma_n)\ .
$$
Now, fix $r\in (0,\sigma)$. Owing to \eqref{c9} and \eqref{c23}, we may pass to the limit as $n\to\infty$ in the above inequality and deduce
$$
B r^{d-1} \partial_r \mathfrak{v}(r) \ge T \int_0^r s^{d-1} \left( \mathfrak{v}(s) - \mathfrak{v}(\sigma) \right) \mathrm{d}s + \frac{\sigma^2 r^{d}}{2d^2} - \frac{r^{d+2}}{2d(d+2)}\ .
$$
As $r\in (0,\sigma)$, it follows from \eqref{c24} that $\partial_r\mathfrak{v}(r)=0$ and $\mathfrak{v}(s)=-1=\mathfrak{v}(\sigma)$ for $s\in (0,r)$, so that we end up with
$$
0=B r^{d-1} \partial_r \mathfrak{v}(r) \ge \frac{r^{d}}{2d^2(d+2)} \left[ (d+2) \sigma^2  - d r^{2} \right] > 0\ ,
$$
and thus a contradiction. We have thus shown that
\begin{equation}
\lim_{n\to\infty} \sigma_n = 0\ . \label{c25}
\end{equation}

Let then $r\in (0,1)$ be arbitrary. Owing to \eqref{c25}, there is $N_r\ge 1$ large enough such that $\sigma_n\in (0,r)$ for $n\ge N_r$. Recalling the definition of $\sigma_n$, this means that, for $n\ge N_r$ and $x\in\unb\setminus \bar{\mathbb{B}}_r$, we have $1+v_n(x)\ge \sqrt{\lambda_n}$ and thus $\lambda_n g(v_n)\le 1$ in $\unb\setminus \bar{\mathbb{B}}_r$. Since $(v_n)_{n\ge N_r}$ is bounded in $H_D^2(\unb)$ by Lemma~\ref{lec1} and $v_n\in \mathcal{S}_r^{\lambda_n}$, we conclude that $(v_n)_{n\ge N_r}$ is bounded in $H_D^4(\unb\setminus \bar{\mathbb{B}}_r)$. The convergence \eqref{c9} then entails 
$$
\omega\in H_D^4(\unb\setminus \bar{\mathbb{B}}_r) \;\;\text{ for all }\;\; r\in (0,1)\ . 
$$
Assume now  $a>0$ for contradiction. Since $\omega\equiv -1$ in $\mathbb{B}_a$, the just established regularity of $\omega$  leads us to
\begin{equation}
0=1+\mathfrak{v}(a) = \partial_r \mathfrak{v}(a) = \partial_r^2 \mathfrak{v}(a) = \partial_r^3 \mathfrak{v}(a) \ . \label{c26}
\end{equation}
Multiplying \eqref{c122} by $\omega$, integrating over $\unb\setminus \bar{\mathbb{B}}_a$ and using \eqref{c12b} and \eqref{c26} give
$$
\int_{\unb\setminus \bar{\mathbb{B}}_a} \left( B |\Delta \omega|^2 + T |\nabla \omega|^2 \right) \mathrm{d}x =0\ .
$$
This implies that $\omega$ is constant in $\unb\setminus \bar{\mathbb{B}}_a$ and contradicts \eqref{c12b} and \eqref{c26}. Consequently,~\mbox{$a=0$} and the proof is complete.
\end{proof}

To finish off the proof of Theorem~\ref{C217}
it just remains to summarize our previous findings. 

\begin{proof}[Proof of Theorem~\ref{C217}]
Recall from Theorem~\ref{T19}, \eqref{AA}, Corollary~\ref{C20}, and \eqref{A0} that $\mathcal{A}=\{(\Lambda(s),U(s))\,:\, s>0 \}$ is a continuous curve with $$\lim_{s\rightarrow 0} (\Lambda(s),U(s))=(0,0)\ ,\quad (\Lambda(s_*),U(s_*))=(\lambda_*,u_{\lambda_*})\ ,$$ and $\mathcal{A}_0=\{(\lambda,u_\lambda)\,:\, \lambda\in (0,\lambda_*)\}$.
Part~(i) of Theorem~\ref{C217} now follows from Lemma~\ref{lec3} and Theorem~\ref{T19}~(iii).

Since $\Lambda(s)\rightarrow 0$ for $s\rightarrow\infty$ as just shown, we find for each $\lambda\in (0,\lambda_*)$ numbers $0<s_1<s_*<s_2$ depending on $\lambda$ such that $\Lambda(s_1)=\Lambda(s_2)=\lambda$. Since $s_1<s_*$, we have $U(s_1)\in\mathcal{A}_0$ and thus $U(s_1)\ge U(s_2)$ by Proposition~\ref{prb2}, Corollary~\ref{C20}, and \eqref{AA}. Moreover, $U(s_1)\not= U(s_2)$ since no bifurcation can occur along the curve $\mathcal{A}_0$ due to the implicit function theorem. This proves part~(ii) of Theorem~\ref{C217}.
\end{proof}

Theorem~\ref{TSSIntroduction} and Remark~\ref{ghoussoub} are now consequences of Theorem~\ref{T19}, Theorem~\ref{T21}, and Theorem~\ref{C217}.

%%%%%%%%%%%%%%%%%%%%%%%%%%%%%%%%%%%%%%%%%%%%%%%%%%%%%%%%%%%%%%%%%%%%
%%%%%%%%%%%%%%%%%%%%%%%%%%%%%%%%%%%%%%%%%%%%%%%%%%%%%%%%%%%%%%%%%%%%
\section{Well-posedness in general domains}\label{sec3}
%%%%%%%%%%%%%%%%%%%%%%%%%%%%%%%%%%%%%%%%%%%%%%%%%%%%%%%%%%%%%%%%%%%%
%%%%%%%%%%%%%%%%%%%%%%%%%%%%%%%%%%%%%%%%%%%%%%%%%%%%%%%%%%%%%%%%%%%%

 We shall now focus on the well-posedness of the dynamic problem. Let us recall that $\Omega$ is an arbitrary smooth domain in $\RR^d$ with $d=1,2$.

%%%%%%%%%%%%%%%%%%%%%%%%%%%%%%%%%%%%%%%%%%%%%%%%%%%%%%%%%%%%%%%%%%%%
%%%%%%%%%%%%%%%%%%%%%%%%%%%%%%%%%%%%%%%%%%%%%%%%%%%%%%%%%%%%%%%%%%%%
\subsection{Well-posedness for the hyperbolic problem}\label{sec31}
%%%%%%%%%%%%%%%%%%%%%%%%%%%%%%%%%%%%%%%%%%%%%%%%%%%%%%%%%%%%%%%%%%%%
%%%%%%%%%%%%%%%%%%%%%%%%%%%%%%%%%%%%%%%%%%%%%%%%%%%%%%%%%%%%%%%%%%%%

In this subsection, we prove Theorem~\ref{ThyperIntroduction}. To lighten the notation, we agree upon setting $\gamma=1$ in the following. We first reformulate \eqref{hyper1}-\eqref{hyper3} as a first-order Cauchy problem and use well-known results on cosine functions (see e.g. \cite[Section 5.5 $\&$ Section 5.6]{A04} for details).  Let us note that the self-adjoint operator $-A =-B\Delta^2+T\Delta$ with domain $H_D^4(\Omega)$ generates an analytic semigroup on $L_2(\Omega)$ with spectrum contained in $[\mathrm{Re}\, z<0]$, its inverse $A^{-1}$ is a compact linear operator on $L_2(\Omega)$, and the square root of $A$ is well-defined. Noticing that $A$ is associated with the continuous coercive form
$$
\langle u,v\rangle =\int_\Omega (B\Delta u\Delta v +T\nabla u\cdot \nabla v)\,\mathrm{d} x\ ,\quad u,v\in H_D^2(\Omega)\ ,
$$
the domain of the square root of $A$ is (up to equivalent norms) equal to  $H_D^2(\Omega)$. Consequently, the matrix operator
$$
\mathbb{A}:=\left(\begin{matrix} 0 & -1\\ A & 1\end{matrix}\right)
$$
with domain $D(\mathbb{A}):=H_D^4(\Omega)\times H_D^2(\Omega)$ generates a strongly continuous group on the Hilbert space $\mathbb{H}:=H_D^2(\Omega)\times L_2(\Omega)$. Writing ${\bf u}_0=(u^0,u^1)$, ${\bf u}=(u,\partial_tu)$, and 
$$
f({\bf u})=\left(\begin{matrix} 0\\-g(u)\end{matrix}\right)\quad\text{with}\quad g(u):= 1/(1+u)^2\ ,
$$
we may reformulate \eqref{hyper1}-\eqref{hyper3} as a Cauchy problem
\begin{equation}\label{CPhyp}
\dot{{\bf u}}+\mathbb{A} {\bf u}=\lambda f({\bf u})\ ,\quad t>0\ ,\qquad {\bf u}(0)={\bf u}_0
\end{equation}
in $\mathbb{H}$ with $\dot{{\bf u}}$ indicating the time derivative. Now, defining for $\kappa\in (0,1)$ the open subset $S(\kappa)$ of $\mathbb{H}$ by
$$
S(\kappa):=\{ u\in H_D^2(\Omega)\,:\, u > -1+\kappa\text{ in } \Omega\}\times L_2(\Omega)\ ,
$$
the function $f:S(\kappa)\rightarrow \mathbb{H}$ is uniformly Lipschitz continuous. A classical argument then entails the following proposition. 

%%%%%%%%%%%%%%%%%%%%%%%%%%%%%%%%%%%%%%%%%%%%%%%%%%%%%%%%%%%%%%%%%%%%
\begin{proposition}\label{P1}
For each ${\bf u}_0\in S(\kappa)$, the Cauchy problem \eqref{CPhyp} has a unique maximal mild solution ${\bf u}=(u,\partial_t u)\in C([0,\tau_m),\mathbb{H})$ for some maximal time of existence $\tau_m=\tau_m({\bf u}_0) \in (0,\infty]$. If $\tau_m<\infty$, then  \begin{equation}\label{blowup0}
\liminf_{t\rightarrow\tau_m}\big(\min_{\bar\Omega} u(t)\big)=-1\ ,
\end{equation}
or
\begin{equation}\label{blowup}
\limsup_{t\rightarrow\tau_m} \|(u(t),\partial_t u(t))\|_{\mathbb{H}}=\infty\ .
\end{equation} 
\end{proposition}
%%%%%%%%%%%%%%%%%%%%%%%%%%%%%%%%%%%%%%%%%%%%%%%%%%%%%%%%%%%%%%%%%%%%

To obtain more regularity on the mild solution  ${\bf u}$, let us consider an initial value in the domain of the generator $-\mathbb{A}$, that is, let ${\bf u}_0 \in \big(H_D^4(\Omega)\times H_D^2(\Omega)\big)\cap S(\kappa)$. Then, since $f$ is Lipschitz continuous,
it follows as in the proof of \cite[Theorem 6.1.6]{Pazy}  that  ${\bf u}: [0,\tau_m)\rightarrow \mathbb{H}$ is Lipschitz continuous and whence differentiable almost everywhere with respect to time. Consequently, we obtain (see also \cite[Corollary 4.2.11]{Pazy}):
 
%%%%%%%%%%%%%%%%%%%%%%%%%%%%%%%%%%%%%%%%%%%%%%%%%%%%%%%%%%%%%%%%%%%%
\begin{corollary}\label{C31}
If ${\bf u}_0 \in \big(H_D^4(\Omega)\times H_D^2(\Omega)\big)\cap S(\kappa)$, then the mild solution ${\bf u}$ is actually a strong solution to \eqref{CPhyp}. That is,~${\bf u}$ is differentiable almost everywhere in time with $\dot{{\bf u}}\in L_1(0,\tau;\mathbb{H})$ for each $\tau\in (0,\tau_m)$ and $$\dot{{\bf u}}(t)=-\mathbb{A} {\bf u}(t)+f({\bf u}(t))$$ in $\mathbb{H}$ for almost every $t\in [0, \tau_m)$.
\end{corollary} 
%%%%%%%%%%%%%%%%%%%%%%%%%%%%%%%%%%%%%%%%%%%%%%%%%%%%%%%%%%%%%%%%%%%%

As a consequence, since ${\bf u}=(u,\partial_t u)$, we deduce under the assumption of Corollary~\ref{C31} that, for each $\tau\in (0,\tau_m)$,
$$
\partial_t^k u\in C([0, \tau_m), H_D^{2-2k}(\Omega))\ ,\quad \partial_t^{k+1} u\in  L_1(0, \tau ; H_D^{2-2k}(\Omega))\ ,
$$ for $k=0,1$ and
\begin{equation}\label{pp}
(B\Delta^2-T\Delta)u=-\partial_t^2 u-\partial_t u-\lambda (1+u)^{-2}\ .
\end{equation}
Since the right-hand side of \eqref{pp} belongs to $L_1(0, \tau ; L_2(\Omega))$, we deduce $u\in L_1(0, \tau ; H_D^4(\Omega))$. Now, testing \eqref{pp} by $\partial_t u\in C([0,\tau_m),L_2(\Omega))$ results in
\begin{equation}\label{ppp}
\frac{1}{2}\frac{\mathrm{d}}{\mathrm{d} t}\left(\int_\Omega\vert\partial_t u\vert^2\,\mathrm{d} x +\int_\Omega \big(B\vert \Delta u\vert^2 +T\vert\nabla u\vert^2\big)\,\mathrm{d} x -2\lambda \int_\Omega\frac{1}{1+u}\,\mathrm{d} x\right)=-\int_\Omega\vert\partial_t u\vert^2\,\mathrm{d} x
\end{equation}
almost everywhere in $[0,\tau_m)$. Assume now that $\tau_m<\infty$ and that \eqref{blowup0} does not occur. Then $(1+u)^{-1}\in L_\infty((0,\tau_m)\times \Omega)$ so that \eqref{blowup} cannot occur as well, whence a contradiction. Consequently, $\tau_m<\infty$ implies \eqref{blowup0}.

To finish off the proof of Theorem~\ref{ThyperIntroduction}, it remains to show that the solution exists globally in time for small $\lambda$ and small initial values. Recall that $H^2(\Omega)$ embeds continuously in $L_\infty(\Omega)$ since $d=1,2$ and let  $c_4>0$ be such that
$$
\| v\|_\infty^2\le c_4\big( B\|\Delta v\|_2^2 + T\|\nabla v\|_2^2\big)\ ,\quad v\in H_D^2(\Omega)\ .
$$
Then we can prove the following result on global existence:

%%%%%%%%%%%%%%%%%%%%%%%%%%%%%%%%%%%%%%%%%%%%%%%%%%%%%%%%%%%%%%%%%%%%
\begin{corollary}\label{C32}
For each $\kappa\in (0,1/2)$, there exists $\lambda_1(\kappa)>0$ such that $\tau_m=\infty$ provided that $\lambda\le \lambda_1(\kappa)$ and
${\bf u}_0 \in \big(H_D^4(\Omega)\times H_D^2(\Omega)\big)\cap S(2\kappa)$ with 
$$
 B\|\Delta u^0\|_2^2+T\|\nabla u^0\|_2^2 +\| u^1\|_2^2\le \frac{(1-2\kappa)^2}{c_4}\ .
$$
\end{corollary} 
%%%%%%%%%%%%%%%%%%%%%%%%%%%%%%%%%%%%%%%%%%%%%%%%%%%%%%%%%%%%%%%%%%%%

\begin{proof}
Since $u^0\ge -1+2\kappa$, we have
$$
T_0:=\sup\{\tau\in (0,\tau_m)\,:\, u(t)\ge -1+\kappa\,,\, t\in [0,\tau)\} >0\ 
$$
and $(1+u(t))^{-1}\le \kappa^{-1}$ for $t\in [0,T_0)$. From \eqref{ppp},
\begin{equation}\label{33}
B\|\Delta u(t)\|_2^2+T\|\nabla u(t)\|_2^2\le  B\|\Delta u^0\|_2^2+T\|\nabla u^0\|_2^2 +\| u^1\|_2^2 +\frac{2\lambda \vert\Omega\vert}{\kappa}
\end{equation}
for $t\in [0,T_0)$ and therefore
$$
\|u(t)\|_\infty^2\le (1-2\kappa)^2+ \frac{2\lambda c_4\vert\Omega\vert}{\kappa} \le (1-\kappa)^2\ ,\quad t\in [0,T_0)\ ,
$$
if $\lambda\le \lambda_1(\kappa)$ with $\lambda_1(\kappa)>0$ sufficiently small.  Consequently, $T_0=\tau_m$ from which $\tau_m=\infty$ by Proposition~\ref{P1}.
\end{proof}

Note that $u\in L_\infty (0,\infty; H_D^2(\Omega))$ and $u(t)\ge -1+\kappa$ for $t\ge 0$ due to   \eqref{33} and $T_0=\infty$.

\begin{remark}\label{3something}
If $\Omega=\unb$, the rotational invariance of  \eqref{hyper1} and the uniqueness of solutions guarantee that $u(t)$ is radially symmetric for each $t\in [0,\tau_m)$ provided that $(u^0,u^1)$ is radially symmetric.
\end{remark}
%%%%%%%%%%%%%%%%%%%%%%%%%%%%%%%%%%%%%%%%%%%%%%%%%%%%%%%%%%%%%%%%%%%%
\begin{remark}
The proof of Theorem~\ref{ThyperIntroduction} is the same if the clamped boundary conditions \eqref{hyper2} are replaced by the pinned boundary conditions \eqref{pinned},
and we obtain a strong solution in this case as well. This improves the existence result for weak solutions in \cite{Gu10}.
\end{remark}
%%%%%%%%%%%%%%%%%%%%%%%%%%%%%%%%%%%%%%%%%%%%%%%%%%%%%%%%%%%%%%%%%%%%

%%%%%%%%%%%%%%%%%%%%%%%%%%%%%%%%%%%%%%%%%%%%%%%%%%%%%%%%%%%%%%%%%%%%
%%%%%%%%%%%%%%%%%%%%%%%%%%%%%%%%%%%%%%%%%%%%%%%%%%%%%%%%%%%%%%%%%%%%
\subsection{Well-posedness for the parabolic problem}\label{sec32a}
%%%%%%%%%%%%%%%%%%%%%%%%%%%%%%%%%%%%%%%%%%%%%%%%%%%%%%%%%%%%%%%%%%%%
%%%%%%%%%%%%%%%%%%%%%%%%%%%%%%%%%%%%%%%%%%%%%%%%%%%%%%%%%%%%%%%%%%%%

To prove Theorem~\ref{TparaIntroduction} we first note that 
$$
g: \{ u\in H_D^{2}(\Omega)\,:\, u\ge -1+\kappa\text{ in } \Omega\}\rightarrow H^2(\Omega)\ ,\quad u\mapsto (1+u)^{-2}
$$
is uniformly Lipschitz continuous and recall that for instance $H^2(\Omega)\hookrightarrow H_D^{1/4}(\Omega)$. We then also recall that the operator $-A=-(B\Delta^2-T\Delta)$ with domain $H_D^4(\Omega)$  generates an analytic semigroup $\{e^{-tA}\,:\, t\ge 0\}$ on $L_2(\Omega)$ with 
$$
\|e^{-tA}\|_{\mathcal{L}(L_2(\Omega))}\le Me^{-\alpha t}\ ,\quad t\ge 0\ ,
$$ 
for some $\alpha >0$. Formulating \eqref{hyper1p}-\eqref{hyper3p} by means of the variation-of-constant formula
$$
u(t)=e^{-tA} u^0- \lambda\int_0^t e^{-(t-s)A}\, g(u(s))\,\mathrm{d} s\ ,\quad t\ge 0\ ,
$$ 
the proof of Theorem~\ref{TparaIntroduction} can be performed by a classical fixed point argument. In particular, the exponential decay of the semigroup entails global existence for small values of $\lambda$ as stated in Theorem~\ref{TparaIntroduction} (iii) (e.g. see \cite[Theorem 1.2 (i)]{ELW1} for details). As in Remark~\ref{3something}, if $\Omega=\unb$, we easily see that $u(t)$ is radially symmetric for each $t\in [0,\tau_m)$ provided that $u^0$ is radially symmetric.

%%%%%%%%%%%%%%%%%%%%%%%%%%%%%%%%%%%%%%%%%%%%%%%%%%%%%%%%%%%%%%%%%%%%
%%%%%%%%%%%%%%%%%%%%%%%%%%%%%%%%%%%%%%%%%%%%%%%%%%%%%%%%%%%%%%%%%%%%
\section{Touchdown in the ball}\label{sec4}
%%%%%%%%%%%%%%%%%%%%%%%%%%%%%%%%%%%%%%%%%%%%%%%%%%%%%%%%%%%%%%%%%%%%
%%%%%%%%%%%%%%%%%%%%%%%%%%%%%%%%%%%%%%%%%%%%%%%%%%%%%%%%%%%%%%%%%%%%

In this last section we return to the case $\Omega=\unb$ and take advantage of the fact that a positive eigenfunction $\phi_1>0$ to the operator $A$ is available, see \cite{LWxx} and Lemma~\ref{L0}. We employ the eigenfunction method as e.g. in \cite{CL81} to show the occurrence of a singularity in finite time as stated in  Proposition~\ref{asterix} and Proposition~\ref{Majestix}.

%%%%%%%%%%%%%%%%%%%%%%%%%%%%%%%%%%%%%%%%%%%%%%%%%%%%%%%%%%%%%%%%%%%%
%%%%%%%%%%%%%%%%%%%%%%%%%%%%%%%%%%%%%%%%%%%%%%%%%%%%%%%%%%%%%%%%%%%%
\subsection{General initial conditions}\label{sec41}
%%%%%%%%%%%%%%%%%%%%%%%%%%%%%%%%%%%%%%%%%%%%%%%%%%%%%%%%%%%%%%%%%%%%
%%%%%%%%%%%%%%%%%%%%%%%%%%%%%%%%%%%%%%%%%%%%%%%%%%%%%%%%%%%%%%%%%%%%

\begin{proof}[Proof of Proposition~\ref{asterix}] 
Recall that Lemma~\ref{L0} ensures the existence of $m_1>0$ and $\phi_1\in C_{D,r}^{4}(\cunb)$ with $\phi_1>0$ in $\unb$, $\|\phi_1\|_{1}=1$, and
$A\phi_1 =m_1 \phi_1$.  
Let $u$ be the maximal solution on $[0,\tau_m)$ to \eqref{hyper1}-\eqref{hyper3} if $\gamma>0$ or \eqref{hyper1p}-\eqref{hyper3p} if $\gamma=0$ corresponding to the initial value $(u^0,u^1)$ and define,
for $t\in [0,\tau_m)$,
\begin{equation}
N(t) := \int_{\unb} \phi_1(x) u(t,x)\ \mathrm{d}x \ge - \int_{\unb} \phi_1(x)\ \mathrm{d}x = -1\ . \label{aq1}
\end{equation}
Assume that
\begin{equation}
\lambda > \frac{4m_1}{27} \ . \label{aq0}
\end{equation}
We multiply \eqref{hyper1} by $\phi_1$, integrate over $\unb$, and use the properties of $\phi_1$, the convexity of $g$, and Jensen's inequality to obtain
\begin{align}
\gamma^2 \frac{\mathrm{d}^2 N}{\mathrm{d}t^2} + \frac{\mathrm{d} N}{\mathrm{d}t} \le & - \int_{\unb} (B \Delta^2 \phi_1 - T \Delta\phi_1) u\ \mathrm{d}x - \lambda g\left( \int_{\unb} \phi_1 u\ \mathrm{d}x \right) \nonumber \\
\le & - m_1 N - \lambda g(N)\ . \label{aq2}
\end{align}
Setting $\chi(z) := m_1 z + \lambda g(z)$ for $z\in (-1,\infty)$, we note that 
\begin{equation*}
\chi \text{ is decreasing in } \left( -1 , z_\lambda \right) \text{ and increasing in } \left( z_\lambda , \infty \right) \ ,
\end{equation*}
where $z_\lambda := (2\lambda/m_1)^{1/3} - 1$. The choice $\lambda>4m_1/27$ guarantees $\chi(z)\ge \chi(z_\lambda)>0$ for $z\in (-1,\infty)$. We then infer from \eqref{aq1} and \eqref{aq2} that, for $t\in [0,\tau_m)$,
\begin{equation*}
\frac{\mathrm{d} N}{\mathrm{d}t}(t) \le -\chi(z_\lambda) \;\;\text{ and }\;\; -1< N(t) \le N(0) - \chi(z_\lambda) t \ ,
\end{equation*}
if $\gamma=0$, respectively
$$
\frac{\mathrm{d} N}{\mathrm{d}t}(t) \le  e^{-t/\gamma^2} \left[ \frac{\mathrm{d} N}{\mathrm{d}t}(0) + \chi(z_\lambda) \right] -\chi(z_\lambda) 
$$
and
$$
-1 < N(t) \le  N(0) + \gamma^2 \left[ \frac{\mathrm{d} N}{\mathrm{d}t}(0) + \chi(z_\lambda) \right] \left( 1 - e^{-t/\gamma^2} \right) - \chi(z_\lambda) t 
$$
if $\gamma>0$. Consequently,
$$
\tau_m \le \left[ 1 + N(0) + \gamma^2 \left( \left| \frac{\mathrm{d} N}{\mathrm{d}t}(0) \right| + \chi(z_\lambda) \right) \right] \frac{1}{\chi(z_\lambda)}<\infty\ .
$$
This completes the proof of Proposition~\ref{asterix}.
\end{proof}

%%%%%%%%%%%%%%%%%%%%%%%%%%%%%%%%%%%%%%%%%%%%%%%%%%%%%%%%%%%%%%%%%%%%
%%%%%%%%%%%%%%%%%%%%%%%%%%%%%%%%%%%%%%%%%%%%%%%%%%%%%%%%%%%%%%%%%%%%
\subsection{Radially symmetric initial data}\label{sec42}
%%%%%%%%%%%%%%%%%%%%%%%%%%%%%%%%%%%%%%%%%%%%%%%%%%%%%%%%%%%%%%%%%%%%
%%%%%%%%%%%%%%%%%%%%%%%%%%%%%%%%%%%%%%%%%%%%%%%%%%%%%%%%%%%%%%%%%%%%

Roughly speaking, the proof of Proposition~\ref{Majestix} proceeds along the same lines of that of Proposition~\ref{asterix} but takes advantage of the properties of the linearization of \eqref{stat1} for $\lambda_*$ described in Lemma~\ref{leb3c} and Proposition~\ref{prb7}.

\begin{proof}[Proof of Proposition~\ref{Majestix}]
Fix $\lambda>\lambda_*$ and let $u$ be the maximal solution on $[0,\tau_m)$ to \eqref{hyper1}-\eqref{hyper3} if $\gamma>0$ or \eqref{hyper1p}-\eqref{hyper3p} if $\gamma=0$ corresponding to the initial value $(u^0,u^1)$. Recall that $\mu_1(u_{\lambda_*})=0$ by Proposition~\ref{prb7} and that there exists a corresponding positive eigenfunction $\phi_*\in C_{D,r}^4(\cunb)$ to the operator $A+\lambda_* g'(u_{\lambda_*})$ according to Lemma~\ref{leb3c}, which we normalize so that $\|\phi_*\|_1=1$. For $t\in [0,\tau_m)$, define
\begin{equation}
M(t) := \int_{\unb} u(t,x) \phi_*(x)\, \mathrm{d}x \ge - \int_{\unb} \phi_*(x)\, \mathrm{d}x = -1\ . \label{q4}
\end{equation}
As in \cite[Theorem~4.1]{Gu10}, we multiply \eqref{hyper1} by $\phi_*$, integrate over $\unb$, and use the equation satisfied by $u_{\lambda_*}$ to obtain
\begin{align*}
\gamma^2 \frac{\mathrm{d}^2 M}{\mathrm{d}t^2} + \frac{\mathrm{d} M}{\mathrm{d}t} = & - \int_{\unb} \phi_* \left( B \Delta^2 u - T \Delta u + \lambda g(u) \right) \mathrm{d}x \\
& + \int_{\unb} \phi_* \left( B \Delta^2 u_{\lambda_*} - T \Delta u_{\lambda_*} + \lambda_* g(u_{\lambda_*}) \right) \mathrm{d}x \\
= &  - \int_{\unb}  (u - u_{\lambda_*}) \left( B \Delta^2 \phi_* - T \Delta \phi_*\right) \mathrm{d}x \\
&  - \int_{\unb} \phi_* \left( \lambda g(u) - \lambda_* g(u_{\lambda_*}) \right) \mathrm{d}x \\
= & \int_{\unb} \phi_* \left[ - \lambda g(u) + \lambda_* g(u_{\lambda_*}) + \lambda_* g'(u_{\lambda_*}) (u-u_{\lambda_*}) \right]\, \mathrm{d}x\ .
\end{align*}
It follows from the convexity of $g$ and Jensen's inequality that 
$$
\lambda_* g(u_{\lambda_*}) + \lambda_* g'(u_{\lambda_*}) (u-u_{\lambda_*}) \le \lambda_* g(u) \;\;\text{ and }\;\; \int_{\unb} g(u) \phi_*\, \mathrm{d}x \ge g(M)\ .
$$
Therefore, owing to the positivity of $\phi_*$ and $\lambda-\lambda_*$,
\begin{align}
\gamma^2 \frac{\mathrm{d}^2 M}{\mathrm{d}t^2} + \frac{\mathrm{d} M}{\mathrm{d}t} \le & \int_{\unb} \phi_* \left( - \lambda g(u) + \lambda_* g(u) \right) \mathrm{d}x \nonumber \\
\le & -(\lambda-\lambda_*) g(M)\ . \label{q6}
\end{align}
Since $(\lambda-\lambda_*) g(M)\ge 0$, a first consequence of \eqref{q6} is that, for $t\in [0,\tau_m)$,
$$
\frac{\mathrm{d} M}{\mathrm{d}t}(t) \le 0 \;\;\text{ and }\;\; M(t)\le M(0)
$$
if $\gamma=0$ or
$$
\frac{\mathrm{d}}{\mathrm{d}t} \left( e^{t/\gamma^2} \frac{\mathrm{d} M}{\mathrm{d}t}(t) \right)\le 0 \;\;\text{ and }\;\; M(t) \le M(0) + \gamma^2 \frac{\mathrm{d} M}{\mathrm{d}t}(0) \left( 1 - e^{-t/\gamma^2} \right)
$$
if $\gamma>0$. In both cases,  for $t\in [0,\tau_m)$,
\begin{equation}
M(t) \le K_{0,\gamma} := M(0) + \gamma^2 \left| \frac{\mathrm{d} M}{\mathrm{d}t}(0) \right|\ . \label{q7}
\end{equation}
Recalling that $g$ is decreasing, we deduce from \eqref{q6} and \eqref{q7} that 
$$
\gamma^2 \frac{\mathrm{d}^2 M}{\mathrm{d}t^2} + \frac{\mathrm{d} M}{\mathrm{d}t} + (\lambda-\lambda_*) g(K_{0,\gamma}) \le 0\ , 
$$
whence, for $t\in (0,\tau_m)$,
$$
M(t) \le M(0) - (\lambda-\lambda_*) g(K_{0,\gamma}) t
$$
if $\gamma=0$ and
$$
M(t) \le M(0) - (\lambda-\lambda_*) g(K_{0,\gamma}) t + \gamma^2 \left( \frac{\mathrm{d} M}{\mathrm{d}t}(0) + (\lambda-\lambda_*) g(K_{0,\gamma}) \right) \left( 1 - e^{-t/\gamma^2} \right)
$$
if $\gamma>0$. We have thus shown that, for $t\in [0,\tau_m)$, 
$$
M(t) \le K_{1,\gamma} - (\lambda-\lambda_*) g(K_{0,\gamma}) t \ ,
$$
 where 
$$ K_{1,\gamma} := M(0) + \gamma^2 \left| \frac{\mathrm{d} M}{\mathrm{d}t}(0) + (\lambda-\lambda_*) g(K_{0,\gamma}) \right|\ .
$$
Recalling that $M(t)\ge -1$ for all $t\in [0,\tau_m)$ by \eqref{q4}, we end up with 
$$
-1 \le K_{1,\gamma} - (\lambda-\lambda_*) g(K_{0,\gamma}) t\ , \qquad t\in [0,\tau_m)\ .
$$
Consequently,
$$
\tau_m \le \frac{1+K_{1,\gamma}}{(\lambda-\lambda_*) g(K_{0,\gamma})} < \infty\ ,
$$
as claimed in Proposition~\ref{Majestix}.
\end{proof}

%%%%%%%%%%%%%%%%%%%%%%%%%%%%%%%%%%%%%%%%%%%%%%%%%%%%%%%%%%%%%%%%%%%%
%%%%%%%%%%%%%%%%%%%%%%%%%%%%%%%%%%%%%%%%%%%%%%%%%%%%%%%%%%%%%%%%%%%%
\appendix
\section{A Brezis-Merle estimate}\label{secbm}
%%%%%%%%%%%%%%%%%%%%%%%%%%%%%%%%%%%%%%%%%%%%%%%%%%%%%%%%%%%%%%%%%%%%
%%%%%%%%%%%%%%%%%%%%%%%%%%%%%%%%%%%%%%%%%%%%%%%%%%%%%%%%%%%%%%%%%%%%
 
We shall prove here in two space dimensions that solutions to the biharmonic equation with homogeneous Dirichlet boundary conditions and right hand sides in $L_1$ belong to $W_q^2$ for any $q\in (1,\infty)$ (a fact which is used in Lemma~\ref{lec1}). This result is strongly reminiscent of the celebrated Brezis-Merle inequality \cite{BM91} stating that solutions to the Laplace equation with homogeneous Dirichlet boundary conditions and right hand sides in $L_1$ belong to $L_q$ for any $q\in (1,\infty)$. The proof of Lemma~\ref{le.app1} below is actually very similar to that of \cite[Theorem 1]{BM91} and is given merely for the sake of completeness. Let us point out that, because of the clamped boundary conditions \eqref{hyper2s}, the result cannot  be deduced  directly from \cite[Theorem 1]{BM91} (in contrast to the case of pinned boundary conditions \eqref{pinned}).

 %%%%%%%%%%%%%%%%%%%%%%%%%%%%%%%%%%%%%%%%%%%%%%%%%%%%%%%%%%%%%%%%%%%%
\begin{lemma}\label{le.app1}
Let $d=2$ and $f\in L_2(\unb)$, $f\not\equiv 0$. Let $w\in H_D^4(\unb)$ be the unique solution to
\begin{equation}
\Delta^2 w = f \;\;\text{ in }\;\; \unb\ , \qquad w=\partial_\nu w = 0 \;\;\text{ on }\;\; \uns\ . \label{bm1} 
\end{equation}
There are $\vartheta_0>0$ and $C_5>0$ independent of $f$ and $w$ such that
$$
\int_{\unb} \exp{\left( \frac{\vartheta_0 \left| D^2 w(x) \right|}{\| f\|_1} \right)}\ \rd x\le C_5 \ ,
$$
where $D^2w$ is the Hessian matrix of $w$.
Furthermore, given $q\in [1,\infty)$, there is $C_6(q)>0$ independent of $f$ and $w$ such that 
$$
\| w \|_{W_q^2} \le C_6(q)\ \|f\|_1\ .
$$
\end{lemma}
%%%%%%%%%%%%%%%%%%%%%%%%%%%%%%%%%%%%%%%%%%%%%%%%%%%%%%%%%%%%%%%%%%%%

\begin{proof}
A classical density argument allows us to assume that $f\in C_0^\infty(\unb)$. Introducing the Green function $G$ associated with the operator $\Delta^2$ subject to homogeneous Dirichlet boundary conditions in $\unb$, it follows from \cite[Theorem~4.7]{GGS10} that there is $K_0>0$ such that
$$
\left| \partial_{x_i} \partial_{x_j} G(x,y) \right| \le K_0  \ln{\left( 2 + \frac{d(y,\uns)}{|x-y|} \right)}\ \min{\left\{ 1 , \frac{d(y,\uns)^2}{|x-y|^2} \right\} }
$$
for $(x,y)\in \cunb\times\cunb$ with $d(\cdot,\uns)$ denoting the distance to $\uns$. From this we deduce 
\begin{equation}
\left| \partial_{x_i} \partial_{x_j} G(x,y) \right| \le K_0  \ln{\left( 2 + \frac{1}{|x-y|} \right)}\ , \qquad (x,y)\in \cunb\times\cunb\ . \label{bm3}
\end{equation}
The solution $w$ to \eqref{bm1} can be written as
$$
w(x) = \int_{\unb} G(x,y) f(y)\ \rd y\ , \qquad x\in\cunb\ ,
$$
which further gives
\begin{equation}
\partial_{x_i} \partial_{x_j} w(x) =  \int_{\unb} \partial_{x_i} \partial_{x_j} G(x,y) f(y)\ \rd y\ , \qquad x\in\cunb\ , \quad 1\le i,j \le 2\ . \label{bm2}
\end{equation}

Now, let $\vartheta\in (0,2/K_0)$. We argue as in the proof of \cite[Theorem~1]{BM91} and use \eqref{bm3}, \eqref{bm2}, the convexity of $z\mapsto e^{\vartheta z}$, and Jensen's inequality to obtain
\begin{align*}
\int_{\unb} \exp{\left( \frac{\vartheta \left| \partial_{x_i} \partial_{x_j} w(x) \right|}{\| f\|_1} \right)}\ \rd x \le & \int_{\unb} \exp{\left( \vartheta \int_{\unb} \left| \partial_{x_i} \partial_{x_j} G(x,y) \right| \frac{f(y)}{\| f\|_1}\ \rd y \right)}\ \rd x \\
\le & \int_{\unb} \int_{\unb} \exp{\left( \vartheta \left| \partial_{x_i} \partial_{x_j} G(x,y) \right|\right)} \frac{f(y)}{\| f\|_1}\ \rd y \rd x \\
\le & \int_{\unb} \frac{f(y)}{\| f\|_1} \int_{\unb} \left( 2 + \frac{1}{|x-y|} \right)^{\vartheta K_0}\ \rd x \rd y \ .
\end{align*} 
For $(x,y)\in\cunb\times\cunb$, we have $x\in \mathbb{B}_2(y)$ and thus
\begin{equation}
\int_{\unb} \exp{\left( \frac{\vartheta \left| \partial_{x_i} \partial_{x_j} w(x) \right|}{\| f\|_1} \right)}\ \rd x \le \int_{\unb} \frac{f(y)}{\| f\|_1} \int_{\mathbb{B}_2(y)} \left( \frac{5}{|x-y|} \right)^{\vartheta K_0}\ \rd x \rd y \le C(\vartheta) \ ,\label{bm4}
\end{equation}
since $-\vartheta K_0>-2$.

\medskip

Consider next $q\in [1,\infty)$. We infer from \eqref{bm4} that
\begin{align*}
\left\| \partial_{x_i} \partial_{x_j} w \right\|_q^q \le & \|f\|_1^q \int_{\unb} \frac{\left| \partial_{x_i} \partial_{x_j} w(x) \right|^q}{\|f\|_1^q}\ \rd x \\
 \le & \|f\|_1^q\  \sup_{z\ge 0}{\left\{ z^q e^{-z/K_0} \right\}} \int_{\unb} \exp{\left( \frac{\left| \partial_{x_i} \partial_{x_j} w(x) \right|}{K_0 \| f\|_1} \right)}\ \rd x \\
\le & 
 C(1/K_0)\ \|f\|_1^q\ \sup_{z\ge 0}{\left\{ z^q e^{-z/K_0} \right\}} \ ,
\end{align*}
which, together with the Poincar\'e inequality, completes the proof.
\end{proof}

%%%%%%%%%%%%%%%%%%%%%%%%%%%%%%%%%%%%%%%%%%%%%%%%%%%%%%%%%%%%%%%%%%%%
\begin{remark}\label{re.app}
According to \cite[Proposition~4.27]{GGS10}, the estimate \eqref{bm3} is valid for an arbitrary smooth domain $\Omega$ of $\RR^2$ (with a constant depending on $\Omega$) so that the validity of Lemma~\ref{le.app1} extends to arbitrary smooth domains of $\RR^2$.
\end{remark}
%%%%%%%%%%%%%%%%%%%%%%%%%%%%%%%%%%%%%%%%%%%%%%%%%%%%%%%%%%%%%%%%%%%%

%%%%%%%%%%%%%%%%%%%%%%%%%%%%%%%%%%%%%%%%%%%%%%%%%%%%%%%%%%%%%%%%%%%%
%%%%%%%%%%%%%%%%%%%%%%%%%%%%%%%%%%%%%%%%%%%%%%%%%%%%%%%%%%%%%%%%%%%%
\section*{Acknowledgments}
%%%%%%%%%%%%%%%%%%%%%%%%%%%%%%%%%%%%%%%%%%%%%%%%%%%%%%%%%%%%%%%%%%%%
%%%%%%%%%%%%%%%%%%%%%%%%%%%%%%%%%%%%%%%%%%%%%%%%%%%%%%%%%%%%%%%%%%%%

This research was done while Ph.L. was enjoying the kind hospitality of the Institut f\"ur Angewandte Mathematik of the Leibniz Universit\"at Hannover. The work of Ph.L. was partially supported by the Centre International de Math\'ematiques et d'Informatique CIMI, Toulouse.

%%%%%%%%%%%%%%%%%%%%%%%%%%%%%%%%%%%%%%%%%%%%%%%%%%%%%%%%%%%%%%%%%%%%
%%%%%%%%%%%%%%%%%%%%%%%%%%%%%%%%%%%%%%%%%%%%%%%%%%%%%%%%%%%%%%%%%%%%
\bibliographystyle{amsplain}
\bibliography{SG4}
%%%%%%%%%%%%%%%%%%%%%%%%%%%%%%%%%%%%%%%%%%%%%%%%%%%%%%%%%%%%%%%%%%%%
%%%%%%%%%%%%%%%%%%%%%%%%%%%%%%%%%%%%%%%%%%%%%%%%%%%%%%%%%%%%%%%%%%%%

\end{document}